\documentclass[11pt,leqno,a4paper,psamsfonts]{amsart}
\usepackage[english]{babel}
\usepackage{amssymb,amsmath,amscd,enumerate,amsthm}
\usepackage[psamsfonts]{amsfonts}
\usepackage{latexsym}
\usepackage{amsbsy}
\usepackage[mathscr]{eucal}
\usepackage{graphicx}
\usepackage{enumerate}
\usepackage{verbatim}

\newcommand{\C}{{\mathbb{C}}}
\newcommand{\cD}{{\mathcal{D}}}
\newcommand{\F}{{\mathcal{F}}}
\newcommand{\R}{{\mathbb{R}}}
\newcommand{\Z}{{\mathbb{Z}}}
\newcommand{\uM}{\underline{M}}
\newcommand{\uF}{\underline{\F}}
\newcommand{\uC}{\underline{\mathcal{C}}}
\newcommand{\uD}{\underline{\mathcal{D}}}
\newcommand{\uphi}{\underline{\varphi}}

\usepackage{color}

\newcommand{\CS}{{\rm CS}}

\newcommand{\svee}{{\scriptscriptstyle \vee}}

\newcommand{\mf}[1]{{\mathfrak{#1}}}
\newcommand{\mr}[1]{{\mathrm{#1}}}
\newcommand{\mc}[1]{{\mathcal{#1}}}
\newcommand{\mb}[1]{{\mathbb{#1}}}
\newcommand{\wt}[1]{{\widetilde{#1}}}

\newcommand{\wh}[1]{{\widehat{#1}}}

\newcommand{\loar}[2]{{#1}\looparrowright_{\wF}{#2}}

\newcommand{\uco}[1]{{
{\underset{\scriptscriptstyle{#1}}\looparrowright}}}

\let\loar\ucon

\newcommand{\iso}{{\overset{\sim}{\longrightarrow}}}

\newcommand{\proan}{{\underleftarrow{\mathrm{An}}}}
\newcommand{\protop}{{\underleftarrow{\mathrm{Top}}}}
\newcommand{\inte}[1]{\overset{\circ}{{#1}}}
\let\CAL=\mathcal%
\def\mathcal#1{{\CAL#1}}%

\newtheorem{teo}{Theorem}[section]
\newtheorem{thmm}{Theorem}

\newtheorem{lema}[teo]{Lemma}

\newtheorem{prop}[teo]{Proposition}
\newtheorem{defin}[teo]{Definition}
\newtheorem{cor}[teo]{Corollary}
\newtheorem{corm}{Corollary}

\newtheorem{obs2}[teo]{Remark}
\newtheorem{recap2}[teo]{Recapitulation}
\newtheorem{ex2}[teo]{Exemple}
\newenvironment{obs}{\begin{obs2}\rm}{\end{obs2}}
\newenvironment{ex}{\begin{ex2}\rm}{\end{ex2}}

\newenvironment{dem}{\begin{proof}[Proof]}{\end{proof}}
\newenvironment{dem2}[1]{\begin{proof}[Proof #1]}{\end{proof}}

\def\bibartp#1#2#3#4#5#6#7#8
{\bibitem{#1} {\sc #2}, {\it #3}, {#4}, {\bf t. #5}, pages { #6} to
{ #7}, {({#8})}}
\def\bibart#1#2#3#4#5#6
{\bibitem{#1} {\sc #2}, {\it #3}, {#4}, {\bf #5}, {({#6})}}
\def\bibliv#1#2#3#4#5
{\bibitem{#1} {\sc #2}, {\it #3}, {#4}, {({#5})}}
\def\bibaart#1#2#3#4
{\bibitem{#1} {\sc #2}, {\it #3}, {#4}}
\usepackage{color}

\title[Topology of singular holomorphic foliations]{Topology of singular holomorphic foliations along a compact divisor}
\date{\today}
\author{David Mar\'{\i}n and Jean-Fran\c{c}ois Mattei}
\thanks{This work was partially supported by grants MTM2008-02294 and MTM2011-26674-C02-01 of  Ministerio de  Econom\'{\i}a y Competitividad 
of Spain / FEDER}

\address{Departament de Matem\`{a}tiques \\ Universitat Aut\`{o}noma de Barcelona \\ E-08193 Bellaterra (Barcelona)\\ Spain} \email{davidmp@mat.uab.es}

\address{Institut de Math\'{e}matiques de Toulouse\\ Universit\'{e} Paul Sabatier\\ 118, Route de Narbonne\\ F-31062 Toulouse Cedex 9, France} \email{jean-francois.mattei@math.univ-toulouse.fr}
\begin{document}
\begin{abstract}
We consider a singular holomorphic foliation $\uF$ defined near a compact curve $\uC$ of a complex surface. Under some hypothesis on $(\uF,\uC)$ we prove that there exists a system of tubular neighborhoods $U$ of a curve $\underline{\mc D}$ containing $\uC$ such that every leaf $L$ of $\uF_{|(U\setminus \underline{\mc D})}$ is incompressible in $U\setminus \underline{\mc D}$.
We also construct  a representation of the fundamental group of the complementary of $\underline{\mc D}$ into a suitable automorphism group, which allows to state the topological classification of
the germ of $(\uF,\uD)$, under the additional but generic dynamical hypothesis of transverse rigidity. In particular, we show that every topological conjugation between such germs of holomorphic foliations can be deformed to extend to the exceptional divisor of their reductions of singularities.
\end{abstract}
\maketitle
%
%

\section{Introduction and main results}
We consider a smooth complex surface $\uM$ endowed with a  holomorphic foliation $\uF$ having isolated singularities and a compact connected holomorphic curve $\uC$.
 To treat in a unified way the local setting we will also allow  the case that $\uC$ reduces to an isolated singular point.
There are two main results in this paper under some hypothesis concerning the pair $(\uF,\uC)$, which we will precise in the sequel:
\begin{enumerate}[(A)]
\item The existence of a fundamental system of neighborhoods of $\uC$ where the leaves of $\uF$ are incompressible in the complementary of an ``adapted'' curve $\underline{\mc D}$ containing $\uC$.
Recall that a subset $A$ of a topological space $V$ is \textbf{incompressible} in $V$ if
the natural inclusion $A\subset V$ induces a monomorphism at the fundamental group level for every choice of the base point in $A$.
\item The construction of a representation of the fundamental group of the complementary of $\underline{\mc D}$ into a suitable automorphism group,
which allows us to state the topological classification of the germ  $(\uF,\uD)$
of $\uF$ along $\uD$. When the curve $\uC$ is smooth and invariant by $\uF$, this object is directly equivalent to the classical holonomy representation of $\pi_{1}(\uC)$ into the automorphisms of a transverse section.
\end{enumerate}

A particular situation of this context occurs when the pair $(\uM,\uC)$ is a resolution of a surface singularity $(S,O)$, see Example~\ref{sup}.
In the general setting it is well known that there exists a composition $E:M\to\uM$  of blow-ups 
such that the curve $\mc C:=E^{-1}(\uC)$ and the foliation $\F:=E^{*}\uF$ satisfy the following properties:
\begin{itemize}
\item $\mc C$ has normal crossings and all its irreducible components $\mc C_i$, $i\in \mf I$ are smooth,
\item two different
irreducible
components of $\mc C$ are disjoint or intersect in a unique point,
\item $\F$ is reduced in the sense of \cite{Corral}, i.e. each singular point of $\F$ has Camacho-Sad index in $\C\setminus\mb Q_{>0}$
and each  component $\mc C_i$ is either $\F$-invariant or $\mc C_i\cap \mr{Sing}(\F)=\emptyset$ and $\F$ is totally transverse to $\mc C_i$.
\end{itemize}
All the notions that we introduce in the sequel are germified along $\cD$ or $\uD$.
By definition the \textbf{isolated separatrix set} of $\F$ is the set $\mc S$ constituted by invariant curves by $\F$, which  are not contained in $\mc C$ and which intersect some $\F$-invariant irreducible component of $\mc C$. The image of the components of $\mc S$ by $E$ are called the \textbf{isolated separatrices} of $\uF$.

Let $\mc G_{\mc C}$ be the \textbf{dual graph} associated to $(M,\mc C)$ having one vertex $\mathbf{s}_i$ for each irreducible component $\mc C_i$ of $\mc C$ and one edge when two irreducible components of $\mc C$ intersect. We also introduce a double weighting $(g_i,\nu_i)$ in each vertex $\mathbf{s}_i$, by giving the genus $g_i=g(\mc C_i)$ and minus the self-intersection $\nu_i=-\mc C_i\cdot \mc C_i$.
It is well known that we can  topologically  recover a tubular neighborhood of $\mc C$ in $M$ by a plumbing procedure from the data given by the dual graph with weights $\mc G_{\mc C}$, see
Section~\ref{plumbing}.

In the sequel we will need to consider a (not necessarily compact) holomorphic curve $\mc D\subset M$ containing $\mc C$.
We define the \textbf{valence} with respect to $\mc D$ of an
irreducible component $D$ of $\mc D$ as the number $v(D)$ of singular points of $\mc D$ lying on $D$.
A \textbf{dead branch} of $\mc D$ is a connected maximal union of irreducible components of $\mc C$ of genus $0$ with valence $2$ with respect to $\mc D$ except for one of them whose valence must be $1$.

Making an additional iterative blowing down process if necessary, without loss of generality we can also assume that
\begin{itemize}
\item
there is no exceptional (i.e. having self-intersection $-1$) $\F$-invariant rational component of $\mc C$
of valence $\le 2$ with respect to $\mc D$.
\end{itemize}

Notice that an irreducible component $D$ (not necessarily compact) of $\mc D$ may be transverse to $\F$. In that case we will say that $D$ is a \textbf{dicritical} component of $\F$.

In order to state our first main result we must introduce some new notions. Denote by $\mc G_{\mc D}$ the dual graph associated to the divisor $\mc D$.
\begin{itemize}
  \item A \textbf{breaking element} of $\mc G_{\mc D}$ is every vertex corresponding to a dicritical component of $\F$ and every edge corresponding to a linearizable singularity of $\F$.
  \item The \textbf{break graph} associated to $(\F,\mc D)$ is the graph
obtained from $\mc G_{\mc D}$ by removing all the breaking elements and the edges whose one of its endpoints is a breaking vertex.
  \item An \textbf{initial component} of $(\F,\mc D)$ is a $\F$-invariant irreducible component $C$ of $\mc C$ such that one of the following situations holds:
\begin{enumerate}[(a)]
\item $g(C)=0$, there is a non-linearizable singular point $p_{0}$ of $\F$ on $C$ and every point $p\in \mr{Sing}(\mc D)\cap C$, $p\neq p_{0}$,   belongs to some dead branch;
 \item $g(C)>0$ and the holonomy of the boundary of every embedded conformal disk in $C$ containing $\mr{Sing}(\mc D)\cap C$ is not linearizable.
\end{enumerate}
\end{itemize}

We introduce two hypothesis on the pair $(\F,\mc C)$. The first one is of local nature and it concerns only the singularities of $\F$. The second one is global and it also  concerns the topology of $\mc C$.
\begin{itemize}
\item[(L)] \textit{The reduced foliation $\F$  has no saddle-nodes and
each singularity $s\in\mr{Sing}(\F)$ having Camacho-Sad index $\lambda_{s}\notin\mb Q$ is linearizable.}
\item[(G)] \textit{Each connected component of the break graph associated to $(\F,\mc C)$ is a tree, which contains at most one vertex corresponding to an initial component $C$ of $\mc C$ of genus $g(C)>0$.}
\end{itemize}

Notice that Condition (L) is generic in the following sense: let $\mc B\subset\C$ be the set of Brjuno numbers, namely those complex numbers $\lambda$ verifying that the germ of every singular foliation defined by a $1$-form of type $(u+\cdots)dv-(\lambda v+\cdots)du$ is always linearizable. It is well known that $\C\setminus\R\subset\mc B$ and that $\R^{-}\setminus\mc B$ has zero measure.

If $\lambda\in \R_{>0}$ then the singularity is a \textbf{node}. Because the reducedness of $\F$ we  have that $\lambda$ is necessarily irrational. If such a singular point $s$ belongs to the strict transform of a (necessarily isolated) separatrix $Z$ of $\uF$ we say that $Z$ is a \textbf{nodal separatrix} of $\uF$ and $s$ a \textbf{nodal singularity}.
The topological specificity of such singularity is the existence, in any small neighborhood of $s$, of a saturated closed set whose complement is an open disconnected neighborhood of the two punctured local separatrices of the node. We call \textbf{nodal separator} such a saturated closed set. A nodal separator of $\uM$ is the image by $E$ of a nodal separator in $M$.

If $D$ is a dicritical component of $\F$ then
for each singular point $s\in \mr{Sing}(\mc D)\cap D$ we consider a conformal closed disk $D_{s}\subset D$ containing $s$ in its interior such that their pairwise intersections are empty. A \textbf{dicritical separator} associated to $D$ is a tubular neighborhood of the closure of $D\setminus\bigcup\limits_{s\in\mr{Sing}(\mc D)\cap D}D_{s}$ which is the total space of a holomorphically trivial disk fibration whose fibers are contained in the leaves of $\F$. A dicritical separator of $\uM$ is the image by $E$ of a dicritical separator of $M$.\\

On the other hand, Condition (G) is not generic and we do not know if the incompressibility of the leaves of $\F$ in the complementary of some $\mc D\supset\mc C$ holds when it is not fulfilled.
Even in the case that Condition (G) holds, the first choice $\mc D=\mc C$ does not work for instance by considering the case that $\mc C$ is the exceptional divisor of the reduction of singularities of a germ of foliation $\uF$ in $\uM=(\C^{2},0)$ because $M\setminus\mc C\cong\C^{2}\setminus\{0\}$ is simply connected.
The next natural choice consists to add to $\mc C$ the isolated separatrices $\mc S$ of $\F$ but this is not enough as the following example shows.

\begin{ex}\label{cusp} Consider the dicritical foliation $\uF$ in $(\mb C^{2},0)$
defined by the rational first integral $\underline{f}(x,y)=\frac{y^{2}-x^{3}}{x^{2}}$ whose isolated separatrix set is the cusp $S=\{y^{2}-x^{3}=0\}$. Let $\uM$ be a Milnor ball for $S$. The composition $E:M\to \uM$ of blow-ups considered in the introduction for this case corresponds to the minimal desingularization of $S$. The exceptional divisor $\mc C=E^{-1}(0)$ has three irreducible components $D_{1},D_{2},D_{3}$ which we numerate according to the order that they appear in the blowing up process. Thus $D_{1}^{2}=-3$, $D_{2}^{2}=-2$ and $D_{3}^{2}=-1$.
The strict transform $\mc S$ of $S$ only meets $D_{3}$. It turns out  that  $D_1$ and $D_2$ are two dead branches composed by a single irreducible component  attached to $D_3$.
Moreover $D_{1}$ a dicritical component. In fact, it is totally transverse to $\F=E^{*}\uF$. Thus, $\mc C\cup\mc S$ do not satisfy Condition (c) in Definition~\ref{adapted}.
On the other hand, it is well-known that if $a$, $b$ and $c$ are meridian loops around $D_{1}$, $D_{2}$ and $D_{3}$ respectively, with common origin, then  $\pi_{1}(\uM\setminus S, \cdot )=\langle a,b,c\,|\, a^{3}=b^{2}=c\rangle$.
We shall see that there exist non-incompressible leaves of $\F$ inside $M\setminus(\mc C\cup\mc S)$. Indeed, looking at the situation after the first blowing-up, we immediately see that there are two types of leaves of $\F$: those that are near to the isolated separatrix set $\mc S$, which are disks minus two points and the others which are diffeomorphic to $\mb D^{*}$. If $L$ is a leaf of the first kind then $\pi_{1}(L)=\langle \alpha_{+},\alpha_{-}|-\rangle$ is a free group of rank $2$.  We claim that we can choose the generators so that the morphism $\imath:\pi_{1}(L)\to\pi_{1}(M\setminus (\mc C\cup\mc S))$ induced by the inclusion is given by $\imath(\alpha_{+})=a$ and $\imath(\alpha_{-})=b^{-1}ab$. It follows that
$\imath(\alpha_{+}^{3}\alpha_{-}^{-3})=a^{3}b^{-1}a^{-3}b=1$ and consequently $L$ is not incompressible in $M\setminus(\mc C\cup\mc S)$. In order to prove the claim we consider the coordinate system $(t,x)$ on $M\setminus (\mc C\cup\mc S)$ induced by the first blowing-up, defined by $E(t,x)=(x,tx)=(x,y)$. We have $f(x,t):=(E^{*}\underline{f})(x,t)=t^{2}-x$ and the restriction of $f$ to $U_{\varepsilon}:=\{|f|= \varepsilon\}$, $0<\varepsilon\ll 1$,  is a locally trivial $C^\infty$-fibration over the  standard circle $\mb S^1_\varepsilon$ of radius $\varepsilon$, whose fiber over $\varepsilon$ is $F_{\varepsilon}:=\C\setminus\{\pm\sqrt{\varepsilon}\}$. Since the pull-back of $U_{\varepsilon}\stackrel{f}{\to}\mb S^{1}_{\varepsilon}$ by the exponential map $\exp:[0,2\pi]\to \mb S^{1}_{\varepsilon}$, $\exp(\theta)=\varepsilon e^{i\theta}$, is trivial, we obtain a trivializing map $\tau:F_{\varepsilon}\times[0,2\pi]\to U_{\varepsilon}$ sending $(z,\theta)$ into $(t,x)=(ze^{i\frac{\theta}{2}},(z^{2}-\varepsilon)e^{i\theta})$. We consider the path $\beta:s\mapsto(z,\theta)=(0,s+\pi)$, $s\in[0,2\pi]$, projecting by $\tau$ into the loop $(t,x)=(0,-\varepsilon e^{is})$ which is a meridian of $D_{2}$. Hence, we can take the generator $b\in\pi_{1}(M\setminus(\mc C\cup\mc S))$ as the homotopy class of $\beta$. Let $z(s)$ be a simple loop in $F_{\varepsilon}$ based on $z=0$ having index $+1$ around $+\sqrt{\varepsilon}$ and index $0$ around $-\sqrt{\varepsilon}$.
We define $\alpha_{-}(s)=(z(s),0)$ and $\alpha_{+}(s)=(z(s),2\pi)$. It is clear that $\alpha_{+}$ is homotopic to $\beta\alpha_{-}\beta^{-1}$ in $F_{\varepsilon}\times[0,2\pi]$. Hence its respective projections by $\tau$ are also homotopic loops in $U_{\varepsilon}$.
Notice that $\tau(\alpha_{-}(s))=(z(s),z^{2}(s)-\varepsilon)$ and $\tau(\alpha_{+}(s))=(-z(s),z^{2}(s)-\varepsilon)$ are meridians around $D_{1}$ so that we can choose the generator $a\in\pi_{1}(M\setminus(\mc C\cup\mc S))\cong\pi_{1}(U_{\varepsilon})$ as the homotopy class of $\tau(\alpha_{+})$.
The fundamental group of the leaf $L$ passing through the point $(t,x)=(0,-\varepsilon)$ is
$\pi_{1}(L)=\langle\alpha_{+},\alpha_{-}|\,-\rangle$ and the images of its generators by $\imath$ are given by $\imath(\alpha_{+})=a$ and $\imath(\alpha_{-})=b^{-1}ab$.

However, if we define $\mc D:=\mc S\cup T\cup \mc C$, where $T$ is the strict transform of  $\{x=0\}$, we can directly see that all the leaves are incompressible in $M\setminus \mc D$. Indeed $T$ meets $D_{1}$ transversely, then
$$\pi _1(M\setminus \mc D,\cdot)=\langle a,b,c\,|\, b^{2}=c,\ [c,a]=1\rangle=\langle a,b\,|\,[a,b^{2}]=1\rangle$$ and the elements $a$ and $b^{-1}ab$ are without relation in this group.
\qed
\end{ex}

Thus, we must make some additional ``holes'' in $M\setminus (\mc C\cup\mc S)$ in order to obtain a bigger fundamental group which could contain the fundamental group of each leaf. This will be done by considering a new divisor $\mc D\supset\mc C\cup\mc S$ obtained by adding some small curves transverse to $\mc C$ satisfying the following technical properties.

\begin{defin}\label{adapted}
We say that a (generally not compact) divisor $\mc D\subset M$ is \textbf{adapted} to $(\F,\mc C)$ if the following conditions hold:
\begin{enumerate}[(a)]
\item the adherence of $\mc D\setminus \mc C$ is a finite union of
conformal disks transverse to $\mc C$ at regular points of $\mc C$ and $\mc D\setminus \mc C$ does not contain any singular point of $\F$;
\item the isolated separatrix set  $\mc S$ is contained in $\mc D$;
\item\label{condc} for every irreducible components $C$ and $D$ of $\mc C$ we have $C\cap D=\emptyset$ provided that $C$ is contained in a dead branch and $D$ is dicritical;
\item\label{condd} if $\mc D=\mc C$,
then it contains at least two irreducible components which do not belong to any dead branch;
\item\label{conde} each connected component of the break graph associated to $(\mc F,\mc D)$ contains at most one vertex corresponding to an initial component of $(\F,\mc D)$
\end{enumerate}
\end{defin}

Adding to $\mc C\cup\mc S$ one non-isolated separatrix over each dicritical component of $\mc C$ having valence $1$ and one transverse curve over certain initial components of genus zero, we obtain a divisor $\mc D$ adapted to $(\F,\mc C)$ provided it fulfills Condition (G):

\begin{prop} If $(\F,\mc C)$ satisfies Condition (G) then there always exists
 a divisor $\mc D$ adapted  to $(\F,\mc C)$.
\end{prop}

In the case that $\mc C$ is the exceptional divisor of the reduction of a germ $\uF$ at $(\C^{2},0)$, in the statement of Corollary~\ref{B} we will precise
the ``minimal''  divisor adapted to $(\F,\mc C)$.

For $A\subset B\subset\uM$ we denote by $\mr{Sat}_{\uF}(A, B)$ the union of all the leaves of	$\F_{|B}$	 passing	through	some	 point	 of $A$.
We fix a plumbing tubular neighborhood $W$ of $\mc C$ in $M$ (see Section~\ref{plumbing}).
The first main result of this paper is the following.

\begin{thmm}\label{main} Let $\mc D$ be a divisor adapted to $(\F,\mc C)$. Assume that $(\F,\mc D)$ satisfies the assumptions (L) and (G) stated below.
Then there exists a fundamental system $(U_{n})_{n\in\mb N}$, $U_{n+1}\subset U_{n}$, of open neighborhoods of $\underline{\mc D}:=E(\mc D)$ in $\uM$ and there exists a smooth holomorphic  curve $\Upsilon\subset M$ transverse to $\F$ having a finite number of connected components, such that for each $n\in\mb N$ the open sets $U_{n}^{*}:=U_{n}\setminus\underline{\mc D}$ and $V^\ast:=E(W)\setminus\underline{\mc D}$ satisfy the following properties:
\begin{enumerate}[(i)]
\item the inclusions $U_{n+1}^{*}\subset U_{n}^{*} \subset V^{*}$ induce  isomorphisms of their funda\-mental groups,
\item every leaf of $\uF_{|U_{n}^{*}}$ is incompressible in $U_{n}^{*}$,
\item each connected component of $Y_{n}^{*}:=E(\Upsilon)\cap U_{n}^{*}$ is a punctured topological disk which is
    incompressible in $U_{n}^{*}$ and
$\mr{Sat}_{\uF}(Y_{n}^{*},U_{n})$ is the complementary in $U_{n}^{*}$ of a finite union of nodal and dicritical separators,
\item\label{1concourbe} there does not exist any path lying on a leaf of $\F_{|U_{n}^{*}}$ with distinct endpoints on $Y_{n}^{*}$ which is homotopic in $U_{n}^{*}$ to a path lying on $Y_{n}^{*}$,
\item the leaf space of the foliation induced by $\F$ in the universal covering space of $U_{n}^{*}$ is a not necessarily Hausdorff one-dimensional complex manifold.
\end{enumerate}
\end{thmm}

\begin{obs}\label{minimal}
It will follow from the proof that a curve $\Upsilon$ satisfying the properties  of Theorem~\ref{main} 
can be constructed in the following way. We choose a vertex in each connected component of the break graph  of $(\F,\mc D)$, a regular point in the corresponding irreducible component of $\mc D$ and we take transversal disks through these points as branches of $\Upsilon$. Hence, the irreducible components of this curve are in one-to-one correspondence with the connected components of the break graph.
\end{obs}

The following corollary of Theorem~\ref{main} completes the main result of \cite{MM}.

\begin{corm}\label{B}
Let $\F$ be a germ of singular holomorphic foliation in $(\C^{2},0)$ which is a generalized curve such that all its singularities after reduction whose Camacho-Sad index is not rational are linearizable.
Then, there exists an open ball $\mb B$ centered at $0$, an analytic curve $\mc Z$ closed in $\mb B$ containing all the isolated separatrices of $\F$, a fundamental system $(U_{n})_{n\in\mb N}$ of neighborhoods of $\mc Z$ in $\mb B$ and a curve $\Upsilon\subset \mb B$, transverse to $\F$ outside the origin, such that the open sets $U_n^\ast:=U_n\setminus \mc Z$ and $V^\ast:=\mb B\setminus \mc Z$ satisfy  Properties (i)-(v)  of Theorem~\ref{main}.
Moreover, if $\F$ is not dicritical then we can take $\mc Z$ as the set of all the separatrices of $\F$. Otherwise, we can take $\mc Z$ as the set of all the isolated separatrices of $\F$ jointly with one non-isolated separatrix of $\F$
for each dicritical component containing a unique singular point of the exceptional divisor of the reduction of $\F$.
\end{corm}

\begin{obs} We point out some issues of each requirement of adapted divisor in Definition~\ref{adapted}:
\begin{enumerate}[(a)]
\item As we have already pointed out, roughly speaking, $W\setminus\mc D$ is obtained from  $W\setminus \mc C$ making some holes in order to that $\pi_{1}(W\setminus\mc D)$ be big enough to contain the fundamental group of each leaf.
\item As it was stated by R. Thom in the seventies, the separatrix set can be viewed as the organization center of the topology of the foliation around a singular point.
Hence it is natural to study the topological embedding of the leaves in the complement of it. In the dicritical case there is an infinite number of separatrices, so the first natural candidate curve to eliminate from the ambient space is the isolated separatrix set.
\item If $\mf m\subset \mc C$ is an invariant dead branch of $\mc D$ then on a neighborhood of $\mf m$, the leaves sufficiently close to $\mf m$ are disks or rational curves.  If moreover $\mf m$ attaches to a dicritical component $D$, then Condition (\ref{condc}) of Definition~\ref{adapted} is not satisfied. Near $D$ the leaves $L$ far away from $\mf m$ are punctured disks with infinite cyclic fundamental group $\mb Z c$, but we can deform the loop $c\subset L$ in the ambient space so that it is conjugated to a loop in a simply connected leaf close to $\mf m$. Hence, in this case we never have the incompressibility of all the leaves.
On the other hand, there exist  counter-examples to the incompressibility of the leaves if we admit some dicritical component contained in a dead branch, as we have already seen in Example~\ref{cusp}, where we have treated in detail the simplest non-trivial dicritical foliation in $(\mb C^{2},0)$ showing this behavior.
\item The radial vector field is a trivial counter-example for the incompressibility of its leaves if Condition (\ref{condd}) of Definition~\ref{adapted} is not satisfied. On the other hand, if $\mc C$ is an $\F$-invariant divisor whose dual graph is a tree and
$(\F,\mc D)$ do not satisfy Condition (\ref{condd}) of Definition~\ref{adapted}, then the intersection matrix $(\mc C_{i}\cdot \mc C_{j})$ can not be negative definite. Indeed, the main result of \cite{Camacho} implies the existence of separatrices (which are necessarily isolated because $\mc C$ is $\F$-invariant) in that case. Hence $\mc D\supset\mc S\cup\mc C\supsetneq\mc C$. Consequently, such divisors do not come from foliations on surface singularities. However, it would be interesting to study the topology of the leaves in this context. The simplest situation occurs when $\mc D=\mc C$ is a chain. Since the restriction of the leaves to a neighborhood of a component of valence $1$ and $2$ are disks and annuli respectively, we deduce that the global leaves in the chain situation are topologically spheres, hence simply connected.
\item Condition (e) of Definition~\ref{adapted} is of technical nature and it comes from the method of construction developed in \cite{MM} which is used in this work.
\end{enumerate}
\end{obs}

\begin{ex}\label{sup}
Let $(S,O)$ the surface singularity $$\{z^{2}=(x^{2}+y^{2})(x^{2}+y^{7})\}\subset(\C^{3},0)$$ considered in \cite{Camacho}. The desingularization $(\uM,\uC)$ of $(S,O)$ is described by a triangular  graph whose vertex represent rational curves having self-intersections $-2$, $-2$ and $-3$, cf. \cite{Laufer}.  After \cite{Wagreich}, the fundamental group $G$ of $S\setminus\{O\}\cong \uM\setminus\uC$ can be presented as
$$G=\langle a,b,c\,|\,cac^{-1}=a^{-3}b^{5},\ cbc^{-1}=a^{-5}b^{8},\ [a,b]=1\rangle$$
and it is solvable. By  the synthesis theorem of \cite{Lins-Neto} there exists  a singular holomorphic foliation $\F$ on $(S,O)$ such that after desingularization defines a singular foliation $\uF$ on $\uM$ whose singularities are reduced and correspond to the three intersection points of the precedent rational curves. By applying the index theorem of \cite{Camacho-Sad} we deduce that the Camacho-Sad index of these singularities belong to the list $\{-\frac{11}{10}\pm\frac{\sqrt{21}}{10},
-\frac{9}{10}\pm\frac{\sqrt{21}}{10},-\frac{3}{2}\pm\frac{\sqrt{21}}{6}\}$. From Siegel and Liouville theorems we deduce that all three singularities are linearizable.
By applying Theorem~\ref{main} we obtain that the fundamental group of each leaf of $\F$ is solvable because it is a subgroup of $G$. Therefore all the leaves of $\F$ are disks or annuli.\qed
\end{ex}

The precedent arguments show a more general result.

\begin{cor}
Let $(S,O)$ be a surface singularity such that the fundamental group of $S\setminus\{O\}$ is solvable. If $\F$ is a singular holomorphic foliation on $(S,O)$ without local separatrices
then all the leaves of $\F$ are disks and annuli.
\end{cor}

\begin{dem}
The classification of configurations with solvable fundamental group given by \cite{Wagreich} and the
hypothesis about the non-existence of local separatrices force all the Camacho-Sad indices to be algebraic numbers, hence  of Brjuno type.
\end{dem}

\bigskip

To deal with the second objective of the paper, the topological classification, we fix the topological type of $\uC$ as embedded divisor in $\uM$, a divisor $\uD$ adapted to $(\uF,\uC)$, a fundamental system $(U_{n})_{n\in\mb N}$ of neighborhoods of $\uC$ fulfilling conditions (i)-(v) of Theorem~\ref{main} and a universal covering $q:\wt U_{0}\to U_{0}\setminus\uD$.
In the sequel we will use the following notations:
$$
\hbox{if}\quad A\subset U_{0} \quad\hbox{then}\quad A^{*}=A\setminus\uD\quad \hbox{and}\quad \wt A=q^{-1}(A^{*}).
$$
Thanks to Property (i) in Theorem~\ref{main}, we can take the restriction of $q$ to $\wt U_{n}$ as universal covering of $U_{n}^{*}$. The deck transformation groups of all these coverings will be identified to $\Gamma:=\mr{Aut}_{q}(\wt U_{0})$. We denote by $\mc Q_{n}$ the leaf space of the foliation 
induced by $\uF$ in $\wt U_{n}$. Clearly the holomorphic natural maps $\mc Q_{n+1}\to\mc Q_{n}$ form an inverse system denoted by $\mc Q^{\uF}$.
As we already pointed out in \cite[\S3]{monodromy} in the local setting,
each deck transformation factorizes through $\mc Q_{n}$ and allows us to consider the notion of monodromy. To this end, we denote by $\underleftarrow{\mc A}$ the category of inverse systems of objects in some category $\mc A$. We refer to \cite{Douady,monodromy} for a precise description of the morphisms in $\underleftarrow{\mc A}$.
\begin{defin}\label{mon}
The monodromy representation of $\uF$ along $\uD$ is the natural morphism of groups
$$\mf m^{\uF}:\Gamma\to\mr{Aut}_{\proan}(\mc Q^{\uF}),$$
where $\proan$ is the category of pro-objects associated to the category $\mr{An}$ of analytic spaces.
\end{defin}

Consider now another foliation $\uF'$ defined in a neighborhood  of a curve $\uC'$
embedded in a surface $\uM'$ and a divisor $\uD'$ adapted to $(\uF',\uC')$.
In order to state our second main result, we need to adapt to our new context some additional notions that we have already considered in \cite{monodromy}:\\

\begin{enumerate}[$\bullet$]
\item We say that a topological conjugation 
between the germs $(\uF,\uD)$ and $(\uF',\uD')$ is \textbf{$\mc S$-transversely holomorphic} if it is transversely holomorphic outside some nodal and dicritical separators.
We have the same notion for conjugations between the germs $(\F,\mc D)$ and $(\F',\mc D')$.
Notice that if there are no dicritical components nor nodal singularities then a  $\mc S$-transversely holomorphic conjugation is just a transversely holomorphic conjugation.\\
\item A \textbf{$\mc S$-conjugation} between the monodromies $\mf m^{\uF}$ and $\mf m^{\uF'}$ consists in $(\varphi,\tilde\varphi,h)$ where $h:\mc Q^{\uF}\to\mc Q^{\uF'}$ is an isomorphism in the category $\protop$,
which is holomorphic outside the subset corresponding to the leaves of some nodal and dicritical separators (we will say that $h$ is a \textbf{$\mc S$-$\proan$ isomorphism}), $\varphi:(\underline{U},\uD)\to (\underline{U}',\uD')$ is a germ of homeomorphism defined in some neighborhoods of $\uD$ and $\uD'$
and $\tilde\varphi$ is a lifting of $\varphi$ to the universal coverings of $\underline{U}\setminus\uD$ and $\underline{U}'\setminus\uD'$ such that
the following diagram commutes
$$\begin{array}{ccc}
\Gamma & \stackrel{\mf m^{\uF}}{\longrightarrow} &  \mr{Aut}_{\proan}(\mc Q^{\uF})\ \subset\  \mr{Aut}_{\protop}(\mc Q^{\uF})\\
\tilde\varphi_{*}
    \downarrow\phantom{\tilde\varphi_{*} } &  &  \phantom{\mr{Aut}_{\proan}(\mc Q^{\uF})\ \subset\   h_{*}}\downarrow h_{*}\\
   \Gamma' &  \stackrel{\mf m^{\uF'}}{\longrightarrow}  & \mr{Aut}_{\proan}(\mc Q^{\uF'})\ \subset\    \mr{Aut}_{\protop}(\mc Q^{\uF'}).
  \end{array}$$
In addition, we say that $(\varphi,\tilde\varphi,h)$ is \textbf{realized over }germs of subsets $\Sigma \subset \underline{M}$ and $\Sigma '\subset \underline{M}'$, if $\varphi(\Sigma)=\Sigma'$ and the following diagram commutes:
$$\begin{array}{ccc}
\wt\Sigma & \stackrel{\wt\varphi_{|\wt\Sigma}}{\longrightarrow} & \wt\Sigma'\\
\downarrow & & \downarrow\\
\mc Q^{\uF} & \stackrel{h}{\longrightarrow} & \mc Q^{\uF'},
\end{array}$$
where the vertical arrows are the natural morphisms.
These notions also apply to
$(\F,\mc D)$ and $(\F',\mc D')$.\\
\item We define the \textbf{cut divisor} $\cD^{\mr{cut}}$ as the disjoint union of the closure of each connected component of the complementary in $\cD$ of nodal singular points and dicritical components of $(\F,\cD)$.
Notice that the dual graph of $\cD^{\mr{cut}}$ is not the break graph of $(\F, \cD)$. These notions are independent.
\\
\item A \textbf{$\mc S$-collection of transversals} for $\uF$ and $\uD$ is a finite collection $\Sigma=\{(\Sigma_{i},p_{i})\}_i$, where each $(\Sigma_{i},p_{i})$ is the image by $E:M\to\uM$ of the germ of a regular curve transverse  to $\F$ at a regular point $p_{i}
\in\cD\setminus\mr{Sing}(\cD)$ not belonging to the exceptional divisor $\mc E$ of $E$, the whole collection satisfying that for each connected component $\cD_{\alpha}^{\mr{cut}}\subset\cD^{\mr{cut}} $ of the cut divisor there exists $i\in\{1,\ldots,m\}$ such that $p_{i}\in\cD_{\alpha}^{\mr{cut}}$.
The existence of a such collection follows from the below lemma whose proof is just adapted from that
of the Strong Camacho-Sad Separatrix Theorem given in \cite{OBRGV10}.
\begin{lema}
There is no irreducible component of $\mc D^{\mr cut}$  contained in the exceptional divisor $\mc E$ of $E$.
\end{lema}
\begin{proof} By contradiction, let $\mc D_{\alpha }^{\rm cut}$ be a component of $\mc D^{\rm cut}$ contained in $\mc E$ and denote by $\mc T$ its dual graph. As in \cite[Section 3]{OBRGV10} the vertices $\mathbf{s}_i$ of $\mc T$ are weighted by the self-intersection of the corresponding component $D_i$ multiplied by $-1$ and to each edge $\mathbf{a}_{ij}$ (joining $\mathbf{s}_i$ and $\mathbf{s}_j$) is associated the pair $(\wp_{ij},\wp_{ji})$, where $-\wp_{ij}$ is equal to the real part Camacho-Sad index $\CS(\F, D_i, s_{ij})$ and $\{s_{ij}\}:=D_i\cap D_j$. At the singular points $s$ of $\mc D$ lying in  the regular part of  $\mc D_{\alpha }^{\mr cut}$ the Camacho-Sad index of $\F$ are not negative real number.
Indeed, it is zero if $s$ is the attaching point of a dicritical component and it is positive if $s$ is a nodal singularity.
Then the  index formulae give the inequalities $\sum_{j}\wp_{ij}\geq D_i\cdot D_i$ and, using the terminology introduced in \cite{OBRGV10}, $\mc T$ is a fair quasi-proper tree. 
We also have the inequalities  $\wp_{ij}\wp_{ji}\leq 1$ and $\mc T$ is well-balanced.  This cannot occur because of Lemma 2.1 of \cite{OBRGV10}, which asserts the no existence of well balanced fair proper tree, is extended to quasi-proper trees in \cite[Section~4]{OBRGV10}.
\end{proof}
\begin{obs} The method developed  in \cite{OBRGV10} immediately give a lower bound for the number of isolated separatrices for dicritical foliations in terms of the number of nodal singularities and dicritical components.
\end{obs}
\item We say that a foliation $\uF$ is \textbf{$\mc S$-transversely rigid} if every topological conjugation between $\uF$ and another foliation $\uF'$ is necessarily $\mc S$-transver\-se\-ly holomorphic. There are many situations in which we have this property. For instance, an extended version of the Transverse Rigidity Theorem  of \cite{Rebelo} already used in \cite{monodromy} asserts that the following condition implies the $\mc S$-transversal rigidity:\\
\begin{enumerate}[(R)]
\item \textit{Each connected component of the cut divisor contains an irreducible component with non-solvable holonomy group.}\\
\end{enumerate}
\item
We call \textbf{$\uD$-extended divisor} every curve $\uD^{+}\supset\uD$ such that $\overline{\uD^{+}\setminus\uD}$ consists in the union of pairs of non-isolated separatrices, one pair for each dicritical separator of $\uF$.\\
\item A germ of homeomorphism $\varphi:(M,\cD)\to(M',\cD')$ is \textbf{excellent} if it satisfies the following properties:
\begin{enumerate}[(a)]
\item outside some neighborhoods of the singular locus of $\cD$ and $\cD'$, $\varphi$ conjugates the smooth disk fibrations $\pi_{i}$ and $\pi_{i}'$ given by Lemma~\ref{lema1};
\item $\varphi$ is holomorphic in a neighborhood of the singular set of $\cD$.
\end{enumerate}
\end{enumerate}

\begin{thmm}\label{main2}
Let $\mc D$ (resp. $\mc D'$) be a divisor adapted to $(\F,\mc C)$ (resp. $(\F',\mc C')$). Assume that $(\F,\mc D)$ and $(\F',\mc D')$ satisfy the assumptions (L) and (G). Then the following statements are equivalent:
\begin{enumerate}[(1)]
\item $(\uF,\uD)$ and $(\uF',\uD')$ are $\mc S$-transversely holomorphic conjugated;
\item $(\F,\mc D)$ and $(\F',\mc D')$ are $\mc S$-transversely holomorphic conjugated by an excellent homeomorphism;
\item  there exists a $\mc S$-conjugation $(\uphi,\wt\uphi,h)$ between the monodromies representations of $\uF$ along $\uD$  and $\uF'$ along $\uD'$, which is
realized over $\mc S$-collec\-tions of transversals, such that:
\begin{enumerate}
\item there exist a $\uD$-extended divisor $\uD^{+}$ such that $\uphi(\uD^{+})$ is a $\uD'$-extended divisor; in addition, for each irreducible component $D$ of $\uD$ we have that $D$ is $\uF$-invariant if and only if $\uphi(D)$ is $\uF'$-invariant;
\item for each singular point $s$ of $\uF$ and each  invariant local irreducible component of $\uD$ at $s$ we have the equality of Camacho-Sad indices $\CS(\uF,D,s)=\CS(\uF',\uphi(D),\uphi(s))$;
\end{enumerate}
\item there exists a $\mc S$-conjugation $(\varphi,\wt\varphi,h)$ between the monodromies representations of $\F$ along $\mc D$  and $\F'$ along $\mc D'$, which is
realized over $\mc S$-collec\-tions of transversals, such that:
\begin{enumerate}
\item for each irreducible component $D$ of $\mc D$ we have that $D$ is $\F$-invariant if and only if $\varphi(D)$ is $\F'$-invariant;
\item for each invariant local irreducible component $D\subset\cD$ at a point $s\in D\cap\mr{Sing}(\cD)$ we have $\CS(\F,D,s)=\CS(\F',\varphi(D),\varphi(s))$;
\item $\varphi$ is excellent.
\end{enumerate}
\end{enumerate}
 Moreover, if $\uF$ satisfies Condition~(R) (more generally if $\uF$ is $\mc S$-transversely rigid) then the precedent properties (1)-(4) are also equivalent to:
\begin{enumerate}[(1')]
\item $(\uF,\uD)$ and $(\uF',\uD')$ are topologically conjugated;
\item $(\F,\cD)$ and $(\F',\cD')$ are topologically conjugated by an excellent homeomorphism.
\end{enumerate}
\end{thmm}

\begin{obs}
The proof of Theorem~\ref{main2} shows in fact that the  conjugations in (1) and (2) (or (1') and (2')) are homotopic in the complementary of the corresponding divisors.
\end{obs}

\begin{corm}\label{corolb}
Let $\F$ be a germ of singular holomorphic foliation in $(\C^{2},0)$ which is a generalized curve such that all its singularities after reduction whose Camacho-Sad index is not rational are linearizable. Assume that $\F$ satisfies Condition (R) below. Let $\F'$ be another germ of singular holomorphic foliation in $(\C^{2},0)$.
Then for every topological conjugation germ $\varphi:(\mb B,\F)\to(\mb B',\F')$ there exists a new topological conjugation germ $\wh\varphi:(\pi^{-1}(\mb B),\pi^{*}\F)\to(\pi'^{-1}(\mb B'),\pi'^{*}\F')$ defined after the reduction processes $\pi$ and $\pi'$ of singularities of $\F$ and $\F'$ such that
\begin{enumerate}
\item $\wh\varphi$ is holomorphic at a neighborhood of $\mr{Sing}(\pi^{*}\F)$,
\item there exist
germs of invariant curves $\mc Z\subset\mb B$ and $\mc Z'\subset\mb B'$
satisfying conclusions of
Corollary~\ref{B} such that $\varphi(\mc Z)=\mc Z'$, $\wh\varphi(\pi^{-1}(\mc Z))=\pi'^{-1}(\mc Z')$ and the restrictions
$\varphi:\mb B\setminus\mc Z\to \mb B'\setminus\mc Z'$ and $\wh\varphi:
\pi^{-1}(\mb B\setminus \mc Z)\to \pi'^{-1}(\mb B'\setminus\mc Z')$ are homotopic.
\end{enumerate}
In particular, the analytic type of the singularities of $\pi^{*}\F$ and its projective holonomy representations are topological invariants of the germ of $\F$ at~$0$.
\end{corm}

Theorem~\ref{main2} with $\uD$ reduced to a point and Corollary~\ref{corolb} generalize Theorems I and II of \cite{monodromy} to the case of dicritical foliations. Moreover, the topological conjugations considered in \cite{monodromy} are assumed to send nodal separatrices into nodal separatrices preserving its corresponding Camacho-Sad indices. In this paper we have used the following result of R. Rosas \cite[Proposition 11]{Rosas} which allows us to eliminate this constraint and to extend our results to general topological conjugations.

\begin{teo}\label{rosas}
Every topological conjugation $\Phi$ between two germs $\F$ and $\F'$ of holomorphic foliations in $(\mb C^{2},0)$ maps nodal  separatrices into nodal separatrices preserving its corresponding Camacho-Sad indices.
\end{teo}

\noindent The idea of the proof is the following.
\begin{enumerate}[(a)]
\item Let $Z$ be a nodal separatrix of $\F$. Any tubular neighborhood  of $Z\setminus\{0\}$ retracts into a $2$-torus $T$ whose first homology group is endowed with a natural $\mb Z$-basis given by monomial coordinates after reduction of singularities of the foliation, cf. \cite[Definition 6.1.2]{monodromy}.
\item Up to a foliated isotopy we can assume that $\Phi$ preserves the $2$-tori $T$ and $T'$ corresponding to $Z$ and $Z':=\Phi(Z)$. It is possible to prove that $\Phi_{*}:H_{1}(T,\mb Z)\to H_{1}(T',\mb Z)$ conjugates its corresponding canonical basis, see \cite[Theorem~10]{Rosas} and \cite[Theorem~6.2.1]{monodromy}.
\item We can canonically identify $T$ and $T'$ with the standard $2$-torus and $\F_{|T}$ and $\F'_{|T'}$ with $1$-dimensional linear irrational foliations.
 It remains to see that the slopes of two linear foliations on the torus are equal once we assume that they are topologically conjugated by a homeomorphism homotopic to the identity.
\end{enumerate}

\section{Localisation}
\subsection{Plumbing}\label{plumbing}
The following result is well known in the literature, cf. for instance \cite{Mumford,Neumann-Eisenbud,Wagreich,Wall}:

\begin{lema}\label{lema1}
There exist an open tubular neighborhood $W$ of $\mc C$ in $M$ and a decomposition $W=\bigcup_{i\in \mf I}W_{i}$ satisfying the following conditions:
\begin{enumerate}[(i)]
\item each $W_{i}$ is a tubular neighborhood of an irreducible component $\mc C_{i}$ of $\mc C$;
\item each $W_{i}$ admits a smooth disk fibration $\pi_{i}:W_{i}\to\mc C_{i}$ over $\mc C_{i}$ whose Euler number $-\nu_{i}$ is the self-intersection of $\mc C_i$;  moreover each non\-empty intersection $\mc C_j\cap W_i$, $i\neq j$, is a fiber of $\pi _i$;
\item there exists a differentiable function $h:W\to\R^{+}$ which is a submersion on $W\setminus \mc C$, such that $h^{-1}(0)=\mc C$ and $\{h^{-1}([0,\varepsilon))\}_{\varepsilon >0}$ is a fundamental system of neighborhoods of $\mc C$, which do not meet the boundary of $\overline{W}$ in~$M$;
\item there exists a simplicial map $\pi:W\to\mc C$ having connected fibres whose restriction to $W_{i}\setminus\bigcup\limits_{j\neq i}W_{j}$ coincides with $\pi_{i}$, $i\in \mf I$.
\end{enumerate}
\end{lema}
Furthermore, we can endow $W$ with a riemannian metric so that the flow of the gradient vector field of $h$ preserves the level hypersurfaces $h=\varepsilon$. In particular, all the neighborhoods $h^{-1}([0,\varepsilon))$ are homeomorphic. Moreover, we can topologically recover $W$ by making the plumbing procedure described in \cite{Mumford,Neumann-Eisenbud} of the fibrations $\pi_{i}:W_{i}\to\mc C_{i}$ obtained from the data given by the dual graph with weights $\mc G$.

\begin{obs}\label{obs14}
We point out some considerations.
\begin{enumerate}[(a)]
\item If additionally the intersection matrix $(\mc C_{i}\cdot \mc C_{j})_{i,j}$ is definite negative then, after Grauert's theorem, there exists a complex structure on the plumbing $W$ such that the quotient $W/\mc C$ becomes a complex surface with an isolated singularity.
\item
The existence of the simplicial map $\pi:W\to\mc C$ having connected fibres implies the existence of a epimorphism
$$\pi_{1}(\partial\overline{W})\to\pi_{1}(\mc C)\cong\pi_{1}(\mc G)\ast\pi_{1}(\mc C_{1})\ast\cdots\ast\pi_{1}(\mc C_{n}),$$
where $\mc G$ is the dual graph associated to $(W,\mc C)$ and $\mc C_{1},\ldots,\mc C_{n}$ are the irreducible components of $\mc C$, cf. \cite{Wagreich}.
\item
We can assume that the fibrations $\pi_{i}:W_{i}\to \mc C_{i}$ are holomorphic in a neighborhood of $\mr{Sing}(\mc C)\cap\mc C_{i}$. Moreover, if $\mc C_{i}$ is a dicritical component of $(\F,\mc C)$ then we can assume that the fibers of $\pi _i$ are the leaves of the restriction $\F_{|W_{i}}$.
\end{enumerate}
\end{obs}

\subsection{Boundary assembly} Let $V$ be a smooth manifold endowed with a regular foliation $\F$ of class $C^{1}$ and let $A$ be an arbitrary subset of $V$. By definition, a leaf of $\F_{|A}$ is a connected component of $L\cap A$, where $L$ is a leaf of $\F$. For every $A\subset V$ we define the \textbf{boundary} of $A$ as $\partial A:=A\setminus\inte{A}$, where $\inte{A}$ is the interior of $A$. The definitions and results of this section are borrowed from \cite{MM}.

\begin{defin}
If $A\subset B\subset V$ we will say that $A$ is \textbf{$1$-$\F$-connected} in $B$ (denoted by $A\loar B$) if for every leaf $L$ of $\F_{|B}$ and for all paths $a:[0,1]\to A$ and $b:[0,1]\to L$ with the same endpoints $m_0$, $m_1$, which are homotopic (with fixed endpoints) in $B$, there exists a path $c:[0,1]\to A\cap L$ with  endpoints $m_0$, $m_1$, which is homotopic to $a$ inside $A$ and to $b$ inside $L$.
\end{defin}

\begin{defin}\label{block}
Let $(V_{i})_{i\in I}$ a finite or numerable collection of submanifolds (with boundary) of $V$ of the same dimension that $V$. We will say that $V_{i}$ is a \textbf{$\F$-adapted block} if it satisfy the following properties:
\begin{enumerate}[(B1)]
\item $\partial V_{i}$ is incompressible in $V_{i}$,
\item $\partial V_{i}$ is a transversely orientable submanifold of $V$ transverse to $\F$,
\item $\partial V_{i}$ is $1$-$\F$-connected in $V_{i}$,
\item every leaf of $\F_{|V_{i}}$ is incompressible in $V_{i}$.
\end{enumerate}
We will say that $V$ is a \textbf{boundary assembly} of the blocks $V_{j}$ if for all $i\in I$ Condition (B1) and the following property hold:
\begin{enumerate}[(B5)]
\item for all different $i,j\in I$ either $V_{i}\cap V_{j}=\emptyset$ or $V_{i}\cap V_{j}$ is a connected component or $\partial V_{i}$ and a connected component of $\partial V_{j}$.
\end{enumerate}
We will say that $V$ is a \textbf{foliated boundary assembly} if each block $V_{i}$ is $\F$-adapted and if $V$ is a boundary assembly of $V_{j}$.
\end{defin}

\begin{teo}[Localisation]\label{loc}
If $V$ is a foliated boundary assembly  of $V_{i}$ then every leaf of $\F$ is incompressible in $V$ and
 for every $I'\subset I$, the union $V'=\bigcup\limits_{i\in I'}V_{i}$ is incompressible and $1$-$\F$-connected  in $V$.
\end{teo}

\begin{obs}\label{pi1assembly}
If $V=\bigcup\limits_{i\in I} V_{i}$ and each block $V_{i}$ satisfy Condition (B5) in previous  Definition~\ref{block}, then we define its dual graph $\mc G_{V}$
by putting one vertex for each element of $I$ and one edge between vertex $i$ and $j$ for each common boundary component of $V_{i}$ and $V_{j}$.
We can give an explicit presentation of the fundamental group of $V$ uniquely from $\pi_{1}(\mc G_{V})$ and the morphisms $\pi_{1}(V_{i}\cap V_{j})\to\pi_{1}(V_{i})$ thanks to the following generalization of the classical Seifert-Van Kampen theorem ($r=0$).
\end{obs}

\begin{prop}
Let $A$ be a connected simplicial complex with connected sub-complex $A_{0}$ and $A_{1}$ such that $A=A_{0}\cup A_{1}$ and $A_{0}\cap A_{1}=B_{0}\sqcup \cdots \sqcup B_{r}$, where each $B_{i}$ is a connected sub-complex of $A_{j}$ for each $i=0,\ldots,r$ and $j=0,1$. Let $\varphi_{ij}:\pi_{1}(B_{i})\to \pi_{1}(A_{j})$ be the morphisms induces by the natural inclusions $B_{i}\subset A_{j}$. Then $\pi_{1}(A)$ is isomorphic to the quotient
$$(\pi_{1}(A_{0})\ast\pi_{1}(A_{1})\ast \mb Z(u_{0})\ast\cdots\ast \mb Z(u_{r}))/K\,,$$
where $K$ is the normal subgroup generated by the relations $u_{0}=1$ and
$$\varphi_{i,0}(b_{i})=u_{i}^{-1}\varphi_{i,1}(b_{i})u_{i},\quad\forall b_{i}\in\pi_{1}(B_{i}),\quad i=0,\ldots,r.$$
\end{prop}
\begin{dem}
See the proof of Proposition~2.1. of \cite{Wagreich} for the case $r=1$. The case $r>1$ is completely analogous.
\end{dem}

\subsection{Decomposition of $\mc D$ and boundary assembly of Milnor tubes}\label{milnor}
We consider the function $h:W\to\R^{+}$ given by Lemma~\ref{lema1} with $h^{-1}(0)=\mc C$. If $f:W\to\C$ is a reduced holomorphic equation of $\overline{\mc D\setminus\mc C}$ then we consider the product $H:=h\cdot |f|$ and we define the
\textbf{open $4$-Milnor tube} of \textbf{height} $\eta>0$ associated to $\mc D$ as $\mc T_{\eta}:=H^{-1}([0,\eta))$.  We also denote
$$\mc T_{\eta}^{*}:=\mc T_{\eta}\setminus\mc D=H^{-1}((0,\eta))$$
and we remark that if $\eta>0$ is small enough then the \textbf{closed $3$-Milnor tube} $\mc M_{\eta}$, defined as the adherence of $H^{-1}(\eta)$ in $\overline{W}$, is transverse to $\partial\overline{W}$. The set of open $4$-Milnor tubes associated to $\mc D$ is a fundamental system of neighborhoods of $\mc D\subset W$. In \cite{Wall} it is shown that there exists a vector field $\xi$ such that $\xi(H)>0$, by
gluing suitable local models with a partition of the unity. The flow of $\xi$ allows to define homeomorphisms between the open $4$-Milnor tubes of different height, provided they are small enough.

For each irreducible component $D$ of $\mc D$ we also consider the disk fibrations $\pi_{D}:W_{D}\to D$ given by Lemma~\ref{lema1} if $D\subset \mc C$ and trivial ones if $D\subset\overline{\mc D\setminus \mc C}$. After Point (c) of Remark~\ref{obs14} we can choose the tubular neighborhoods $W_{D}$ and the fibrations $\pi_{D}$ in such a way that for each singular point $s\in\mr{Sing}(\mc D)$ the following properties hold:
\begin{enumerate}[(a)]
\item If $\{s\}=D\cap D'$ then $W_{s}:=W_{D}\cap W_{D'}$ admits holomorphic local coordinates $(x_{s},y_{s}):W_{s}\stackrel{\sim}{\to}\mb D_{2}\times\mb D_{2}$ such that the germ of $\F$ at $s$ is given by a $1$-form of the following type:
\begin{itemize}
\item $x_{s}dy_{x}-\lambda_{s}y_{s}dx_{s}$ with $\lambda_{s}\in\C$, if $s$ is a linearizable singularity;
\item $x_{s}dy_{s}-(\lambda_{s}y_{s}+x_{s}y_{s}(\cdots))dx_{s}$ with $\lambda_{s}\in\mb Q_{<0}$, if $s$ is a resonant singularity;
\item $dx_{s}$ (resp. $dy_{s}$) if $D$ (resp. $D'$) is a dicritical component of $(\F,\mc D)$.
\end{itemize}
\item $D\cap W_{s}=\{y_{s}=0\}$, $D'\cap W_{s}=\{x_{s}=0\}$ and the restrictions of $\pi_{D}$ and $\pi_{D'}$ to $W_{s}\cap\{|x_{s}|<\frac{3}{2}\}$ and $W_{s}\cap\{|y_{s}|<\frac{3}{2}\}$ coincide with $(x_{s},y_{s})\mapsto x_{s}$ and $(x_{s},y_{s})\mapsto y_{s}$ respectively.
\end{enumerate}
For each irreducible component $D$ of $\mc D$ we denote $\Sigma_{D}:=\mr{Sing}(\mc D)\cap D$ and
\begin{equation}\label{Ds}
D_{s}:=D\cap\{|x_{s}|\le 1,\,|y_{s}|\le 1\}\quad \hbox{for}\quad s\in \Sigma_{D}\,.
\end{equation}
For each irreducible $\F$-invariant component $D$ of $\mc D$ of genus $g(D)>0$ we fix a smooth real analytic curve $\Gamma_{D}$ which is the boundary of a closed conformal disk $D_{\Gamma_{D}}$ containing $\Sigma_{D}$ such that the holonomy of $\Gamma_{D}$ is linearizable, provided that $D$ is not an initial component, see the introduction.
Notice that $\Sigma_{D}\neq\emptyset$ because of Condition (e) in Definition~\ref{adapted}.
If  $D$ contains a unique singular point $s$ of $\mc D$ then we shall take $\Gamma_{D}=\partial D_{s}$ and $D_{\Gamma_{D}}=D_{s}$. Otherwise we can assume that every two closed disks ${D}_{s}$ and ${D}_{s'}$, $s,s'\in\Sigma_{D}$, are disjoint and contained in the open disk ${D}_{\Gamma_{D}}\setminus\partial D_{\Gamma_{D}}$ when $g(D)>0$.
We also denote
\begin{equation}\label{g=0}
D^{*}:=\overline{D\setminus\bigcup\limits_{s\in\Sigma_{D}}D_{s}}, \quad \text{if}\quad g(D)=0
\end{equation}
and
\begin{equation}\label{g>0}
D^{*}:=\overline{D_{\Gamma_{D}}\setminus\bigcup\limits_{s\in\Sigma_{D}}D_{s}}, \quad \text{if}\quad g(D)>0.
\end{equation}
Consider the union $\mf J\subset\mc D$ of all the Jordan curves of the form $\Gamma_{D}$ with $g(D)>0$ and all the curves $\partial D_{s}$ with $s\in D\cap \mr{Sing}(\mc D)$. Let $\mf A$ be the set of \textbf{elementary blocks of $\mc D$} defined as the adherence of the connected components of $\mc D\setminus\mf J$. There exists an \textbf{uniformity height} $\eta_{1}>0$ such that for all $\eta\in (0,\eta_{1}]$ the set $\{\mc T_{\eta}(A)\}_{A}$ composed by the adherence of the connected components of
$$\mc T_{\eta}\setminus\bigcup\limits_{D\subset\mc D}\pi_{D}^{-1}(\mf J\cap D)$$
is in one to one correspondence with $\mf A$. More precisely, for each $A\in\mf A$ there is a unique connected component of $\mc T_{\eta}\setminus\bigcup\limits_{D\subset\mc D}\pi_{D}^{-1}(\mf J\cap D)$ containing $A\subset\mc D$ and whose adherence we denote by $\mc T_{\eta}(A)$.
Notice that for each elementary block $A\subset \mc D$ we can construct a vector field $\xi_{A}$ whose flow induces deformation retracts between $(\mc T_{\eta}^{*}(A),\partial\mc T_{\eta}^{*}(A))$ and  $(\mc T_{\eta_{1}}^{*}(A),\partial\mc T_{\eta_{1}}^{*}(A))$ for all $\eta\in(0,\eta_{1}]$, see Theorem~5.1.5 and Proposition~9.3.2 of \cite{Wall}.
If $B=\cup_{i}A_{i}\subset \mc D$ is an arbitrary union of elementary blocks of $\mc D$
we also adopt the following convenient notation
\begin{equation}\label{TdeB}
\mc T_{\eta}(B):=\bigcup_{i}\mc T_{\eta}(A_{i})\quad \hbox{and} \quad\mc T_{\eta}^{*}(B):=\mc T_{\eta}(B)\setminus \mc D\,.
\end{equation}

\begin{defin}
We will say that an inclusion $\imath:A\subset B$ between two subspaces of a topological space is \textbf{rigid} if $\imath_{*}:\pi_{1}(A,p)\stackrel{\sim}{\to}\pi_{1}(B,p)$ is an isomorphism for all $p\in A$. We will say that $\imath$ is $\partial$-\textbf{rigid} if $\partial A\subset\partial B$ and the two inclusions $A\subset B$ and $\partial A\subset\partial B$ are rigid. Recall that $\partial A=A\setminus\inte{A}$.
\end{defin}

\begin{prop}\label{rig2}
Consider a subset $\mc B\subset \mc T_{\eta_{1}}^{*}$. If for each elementary block $A$ of $\mc D$ the inclusion $\mc B\cap\mc T_{\eta_{1}}^{*}(A)\subset\mc T_{\eta_{1}}^{*}(A)$ is $\partial$-rigid, then the inclusion $\mc B\subset \mc T_{\eta_{1}}^{*}$  is also rigid. In particular the inclusion $\mc T_{\eta}^{*}\subset\mc T_{\eta_{1}}^{*}$ is rigid for all $\eta\in(0,\eta_{1}]$.
\end{prop}
\begin{proof}
This assertion follow immediately  from Remark~\ref{pi1assembly} and the following (trivial) result.
\end{proof}

\begin{lema}\label{rigid}
Let $A\subset B\subset C$ be topological spaces. If two of the three inclusions $A\subset B$, $B\subset C$ and $A\subset C$ are ($\partial$-)rigid then the third one is also ($\partial$-)rigid.
\end{lema}

Notice that the collection $\{\mc T_{\eta}^{*}(A)\}_{A\in\mf A}$ does not define a boundary assembly of $\mc T_{\eta}^{*}$ because Condition (B1) in Definition~\ref{block} is not always verified. More precisely, if $C$ is an irreducible component of $\mc C$ having genus $0 $ and valence $1$ then the boundary of $\mc T_{\eta}^{*}(C)$ is not incompressible. This situation leads us to consider bigger blocks of $\mc D$ as we have already done in \cite{MM}.

\begin{defin}\label{defblocks}
The \textbf{fundamental blocks} of $\mc D$ are the  unions of elementary blocks of $\mc D$ described below:
\begin{enumerate}[(a)]
\item For each $\F$-invariant irreducible component $D$
of $\mc D$ not contained in a dead branch, we consider the \textbf{aggregate block} defined as
$$\mf m_{D}\cup D^{*}\cup \left(\bigcup_{s\in \mf m_D\cap D}{D}_s\right)\,,$$
where $\mf m_{D}$ is the union of all the dead branches meeting $D$,
$D^{*}$ is defined by Equations (\ref{g=0}) or (\ref{g>0}), and $D_s$ is given by (\ref{Ds}).
\item For each singularity $s\in\mr{Sing}(\mc D)$ belonging to different irreducible components $D$ and $D'$ of $\mc D$, we consider the \textbf{singularity block} ${D}_{s}\cup{D_{s}'}$ provided that $s$ do not belong to any dead branch.
\item For each $\F$-invariant irreducible component $D\subset \mc C$ of genus $g(D)>0$, we consider the \textbf{genus block} $\overline{D\setminus D_{\Gamma_{D}}}$.
\item For each dicritical irreducible component $D$ of $\F$, we consider the \textbf{dicritical block}
$$D\cup\bigcup\limits_{(s,D')\in\mf K_{D}}{D_{s}'},$$
where $\mf K_{D}$ is the set of pairs $(s,D')$ constituted by a singular point $s$ of $\mc D$ lying on $D$ and the irreducible component $D'\neq D$ of $\mc D$ meeting
$D$ at~$s$.
\end{enumerate}
An \textbf{initial block} of $\mc D$ is either an aggregate block containing a single singular point of $\mc D$ which do not belong to any dead branch, or a genus block associated to an initial component $D$ of $\mc D$ of genus $g(D)>0$ such that the holonomy of $\Gamma_{D}$ is not linearizable.
A \textbf{breaking block} of $\mc D$ is either a singularity block associated to a linearizable singular point or a dicritical block.
\end{defin}
\begin{prop}\label{BA}
For every $\eta\in(0,\eta_{1}]$, $\mc T_{\eta}^{*}$ is a boundary assembly of the blocks  $\{\mc T_{\eta}^{*}(B)\}_{B\in\mf B}$ defined by (\ref{TdeB}), where $\mf B$ is the set of fundamental blocks of $\mc D$.
\end{prop}
The proof of this proposition will be based on explicit descriptions of the fundamental groups of $\mc T_{\eta}^{*}(B)$, by generators and relations. But before we must  give some preliminary information about the topology of tubular neighborhoods of dead branches.
Let $\mf m=\bigcup\limits_{i=1}^{\ell} D_{i}$ be a $\F$-invariant dead branch with $v(D_{1})=1$, $v(D_{i})=2$ for $i=2,\ldots,\ell$. For each $j=1,\ldots,\ell$ the  intersection matrix
of $\bigcup\limits_{i=1}^{j}D_{i}$ is
$$A_{j}=\left(\begin{array}{ccccc} e_{1} & 1 & 0 &\cdots & 0\\ 1 & e_{2} & 1 &\ddots & \vdots\\ 0 & 1 & e_{3} & \ddots & 0\\ \vdots & \ddots & \ddots & \ddots & 1\\ 0 & \cdots & 0 & 1 & e_{j}\end{array} \right)$$
whose determinant is denoted $\delta_{j}=\det(A_{j})$.
Assume that
the attaching component $C$ of $\mf m$ is also $\F$-invariant according to Condition (c) in Definition~\ref{adapted} and that  $\mf m$ can not be blow-down according to the initial assumptions stated in the introduction. Let $d_{i}\in \pi_{1}(\mc T_{\eta}^{*}(\mf m))$ be a meridian of $D_{i}$ and let $c\in\pi_{1}(\mc T_{\eta}^{*}(\mf m))$ be a meridian of  $C$.
\begin{lema}\label{p2}
There exist coprime positive integers $p\ge 2$ and $q\ge 1$ such that
$d_{\ell}^{p}=c^{q}$.
\end{lema}
\begin{dem2}{of Lemma \ref{p2}} By assumption $C\cup\mf m$ is $\F$-invariant and the singularities of $\F$ are reduced and they are not saddle-nodes. Then classically the Camacho-Sad indices
$$\lambda_{i}=\CS(\F,D_{i},s_{i})\,,\qquad i=1,\ldots,\ell$$
are rational strictly negative numbers.
By Camacho-Sad formula follows that $\lambda_{i}=e_{i}-\frac{1}{\lambda_{i-1}}$.
On the other hand, by developing the determinant of $A_{i+1}$ by the last row, we have the equality $\delta_{i+1}=e_{i+1}\delta_{i}-\delta_{i-1}$.
We claim that $\lambda_{i}=\frac{\delta_{i}}{\delta_{i-1}}$, for $i=2,\ldots,\ell$. Indeed, this is trivially the case for $i=2$ and the inductive step $i\Rightarrow i+1$:
$$\lambda_{i+1}=e_{i+1}-\frac{1}{\lambda_{i}}=e_{i+1}-\frac{\delta_{i-1}}{\delta_{i}}=\frac{e_{i+1}\delta_{i}-\delta_{i-1}}{\delta_{i}}=\frac{\delta_{i+1}}{\delta_{i}}$$
completes the proof of the claim. Since $\lambda_{i}<0$ for all $i=1,\ldots,\ell$ and $\delta_{1}=e_{1}<0$, it follows that $(-1)^{i}\delta_{i}>0$ and, by Silvester's criterion, the matrix $A_{\ell}$ is definite negative.
We take $p=(-1)^{\ell}\delta_{\ell}\ge 1$ and $q=(-1)^{\ell-1}\delta_{\ell-1}\ge 1$.
By Grauert's criterion, $\mf m$ can be blow-down if and only if $A_{\ell}$ is definite negative and $\delta_{\ell}=\pm 1$. Hence $p\ge 2$ by the assumption on $\mf m$. Moreover we have
$$\gcd(p,q)=\gcd(\delta_{\ell},\delta_{\ell-1})=
\gcd(\delta_{\ell-1},\delta_{\ell-2})=\cdots=\gcd(\delta_{2},
\delta_{1})=\gcd(e_{1},-1).$$
Hence $\gcd(p,q)=1$.
It only remains to prove that $d_{\ell}^{p}=c^{q}$.
This equality follows directly from the fact that $(A_{\ell}^{-1})_{\ell\ell}=\frac{\delta_{\ell-1}}{\delta_{\ell}}$ and from the relation $A_{\ell}v+w=0$
in $H_{1}(\mc T_{\eta}^{*}(\mf m),\mb Z)$, where $v=([d_{1}],\cdots,[d_{\ell}])^{t}$ and $w=(0,\ldots,0,[c])^{t}$.  This relation being a matrix reformulation of the Camacho-sad index formulae along the components of $\mf m$.
\end{dem2}

The following result is well-known in combinatorial group theory.

\begin{lema}\label{unicity} If $p_{i}\ge 2$ and $q_{i}\ge 1$ are coprime integers then
every element $\gamma$ of the group $\Gamma$ presented by $$\langle c,d_{1},\ldots,d_{m}\,|\,[c,d_{i}]=1,\, d_{i}^{p_{i}}=c^{q_{i}},i=1,\ldots,m\rangle$$ can be written in a unique way as  $\gamma=u_{1}\cdots u_{r}c^{s}$ with $u_{i}=d_{j_{i}}^{\varepsilon_{i}}$, $0\le\varepsilon_{i}<p_{i}$ and $s\in\mb Z$.
\end{lema}

\begin{dem2}{of Proposition~(\ref{BA})}
By construction the family  $\{\mc T_{\eta}^{*}(B)\}_{B\in\mf B}$ satisfy property (B5) in Definition~\ref{block}. In order to check Condition (B1) for each block $\mc T_{\eta}^{*}(B)$, we will distinguish four cases according to the type of $B\in\mf B$:
\begin{enumerate}[(a)]
\item If $B$ is the aggregated block associated to an $\F$-invariant irreducible component $D$ of $\mc D$ then, after \cite{Wagreich}, we obtain a presentation of $\pi_{1}(\mc T_{\eta}^{*}(B))$ by considering the generators
$ a_{1},\ldots,a_{g},b_{1},\ldots,b_{g},c,d_{1},\ldots,d_{v}$ and the relations
\begin{equation}\label{relations}
[c,*]=1,\quad c^{\nu}\cdot\prod\limits_{i=1}^{g}[a_{i},b_{i}]\cdot\prod\limits_{j=1}^{v}d_{j}=1,\quad d_{k}^{p_{k}}=c^{q_{k}},\ k=1,\ldots,m\le v,
\end{equation}
where $g$, $v$ and $\nu$ are respectively the genus, the valence and the self-intersection of  $D$, $m$ is the number of dead branches contained in $B$ and $p_{k},q_{k}$ are the positive coprime integers given by Lemma~\ref{p2}. Each connected component of the boundary of $\mc T_{\eta}^{*}(B)$ is a torus whose fundamental group is $\langle c,d_{j}|[c,d_{j}]=1\rangle$, $j=m+1,\ldots,v$. The incompressibility of the boundary is equivalent to the following implication
\begin{equation}\label{imp}
(j>m\textrm{ and } d_{j}^{\alpha}c^{\beta}=1\textrm{ in }\pi_{1}(\mc T_{\eta}^{*}(B)))\Longrightarrow \alpha=\beta=0,
\end{equation}
which is trivially true if $m=v$. Hence, in the sequel we will assume that $m\le v-1$.
Notice that $\pi_{1}(\mc T_{\eta}^{*}(B))=\Gamma\ast_{C}G$ where
$\Gamma$ is the group considered in Lemma~\ref{unicity},  $G$ is defined by
$$G:=\langle a_{1},\ldots,a_{g},b_{1},\ldots,b_{g},c,d_{m+1},\ldots,d_{v-1}|\,[c,*]=1\rangle\cong\mb Z^{* 2g+v-m-1}\oplus\mb Z$$
and $C=\langle c|\,-\rangle\cong\mb Z$. Trivially $C$ injects in $G$.
On the other hand, because $p_{i}\ge 2$ for $i=1,\ldots,m$,  $C$ also injects into $\Gamma$.  Seifert-Van Kampen Theorem implies that $G$ can also considered as a subgroup of
$\pi_{1}(\mc T_{\eta}^{*}(B))$. Thus, for $j\le v-1$ implication (\ref{imp}) can be considered in the subgroup $G$, where it is trivially true. It only remains to treat the case of $j=v$. But  $d_{v}^{\alpha}c^{\beta}$ is equal to $\left(\prod\limits_{i=1}^{g}[a_{i},b_{i}]\prod\limits_{j=1}^{v-1}d_{j}\right)^{-\alpha}c^{\beta-\nu\alpha}$ and this expression  can not be simplified using the relations (\ref{relations}), if $g>0$ or $v-m\ge 2$ provided $(\alpha,\beta)\neq(0,0)$. It remains to consider the situation $g=0$ and $v-m=1$. In this case  $G=C$ and the element of $\Gamma$ given by $$d_{v}^{\alpha}c^{\beta}=\underbrace{(d_{1}\cdots d_{m})\cdots(d_{1}\cdots d_{m})}_{-\alpha}\, c^{\beta-\nu\alpha}$$ is written in the unique reduced form stated in Lemma~\ref{unicity}. Consequently it is trivial if and only if $\alpha=\beta=0$.

\item If $B$ is a singularity block then $\mc T_{\eta}^{*}(B)\cong T\times [0,1]$ and $\partial \mc T_{\eta}^{*}(B)\cong T\times \{0,1\}$, so that
 each connected component of its boundary is incompressible.
\item If $B$ is a genus block ($g>0$) then $$\pi_{1}(\mc T_{\eta}^{*}(B))\cong \langle a_{1},\ldots,a_{g},b_{1},\ldots,b_{g},c\,|\, [c,*]=1\rangle$$ contains $\pi_{1}(\partial\mc T_{\eta}^{*}(B))=\langle \prod\limits_{i=1}^{g}[a_{i},b_{i}],c\,|\,[c,*]=1\rangle$.
\item If $B$ is the dicritical block associated to some dicritical component $D$ of $\mc D$ of genus $g\ge 0$ and valence $v\ge 1$ then $D$ is not adjacent to any dead branch of $\mc D$
by Condition (c) in Definition~\ref{adapted}
 and consequently $\pi_{1}(\mc T_{\eta}^{*}(B))$ is the group $G$ considered in case (a) taking $m=0$.
Each connected component of $\partial\mc T_{\eta}^{*}(B)$ is a torus whose fundamental group $\langle d_{j},c\,|\,[c,d_{j}]=1\rangle$ injects into $G$.
\end{enumerate}
\end{dem2}

\subsection{Existence of adapted blocks} In order to control the topology of the foliated blocks that we will construct in Section~\ref{construction} we must consider the notions of size and roughness of a suspension type subset introduced in \cite{MM}.
First of all we recall the notion of suspension type subset. Let $P$ be a regular point of $\F$ lying on an irreducible component $D$ of $\mc D$, let $\Delta$ be a subset contained in the fibre $\pi_{D}^{-1}(P)$ and let $\mu$ be a path contained in $D$ with origin $P$.

\begin{defin}
The \textbf{suspension} of $\Delta$ over $\mu$ along the fibration $\pi_{D}$ is the union
$$\mb V_{\Delta,\mu}:=\bigcup\limits_{m\in\Delta}|\mu_{m}|,$$
where $\mu_{m}$ denotes the path of origin $m$ lying on the leaf of $\F$ passing through $m$ which lifts the path $\mu$ via $\pi_{D}$, i.e. $\pi_{D}\circ \mu_{m}=\mu$ and $\mu_{m}(0)=m$.
\end{defin}

This notion is well defined provided $\Delta$ is small enough. In \cite{MM} we have also introduced the notion of \textbf{roughness} $\mf e(\xi)$ of an oriented curve $\xi\subset\C^{*}$. Here we will say that $\Omega\subset\C$ is of infinite roughness if $\Theta=\overline\Omega\setminus\inte{\Omega}$ is not a piecewise smooth curve. Otherwise we will define the \textbf{roughness}  of $\Omega$  as $\mf e(\Omega)=\inf\{\mf(\Theta^{+}),\mf(\Theta^{-})\}$, where $\Theta^{+}$ and $\Theta^{-}$ are two curves of opposite orientations parameterizing $\Theta$.
The finiteness of the roughness is equivalent to the starlike property with respect to the origin.

Since every open Riemann surface is Stein, each fibration $\pi_{D}:W_{D}\to D$ is analytically trivial over every open set $D'\subsetneq D$.
Fix on $W':=\pi_{D}^{-1}(D')$ a trivializing coordinate
 $z_{D'}:W'\to\C$, i.e. $(z_{D'},\pi_{D})$ is a biholomorphism from $W'$ onto the product of the unit disk of $\C$ times $D'$; we define the \textbf{roughness} of a subset $E$ of $W'$ with respect to $z_{D'}$ as
$$\mf e_{z_{D'}}(E):=\sup\{\mf e(z_{D'}(E\cap\pi_{D}^{-1}(m))),\, m\in D\}\in\mb R^{+}\cup\{\infty\}.$$
We also define the \textbf{size} of $E$ with respect to $z_{D'}$ as
$$\|E\|_{z_{D'}}:=\max\{|z_{D'}(m)|,\, m\in E\},$$
and we denote $\mf c(\cdot)=\max\{\mf e_{z_{D'}}(\cdot),\|\cdot\|_{z_{D'}}\}$ called control function. \\

Now we present an existence theorem of  $\F$-adapted blocks having controlled size and roughness, which we will prove in next section.
In Section~\ref{sec4} we shall prove Theorem~\ref{main} by gluing inductively these $\F$-adapted blocks and using Localization Theorem~\ref{loc}.
We keep the notation $\mc T_{\eta}^{*}(A)$ for the blocks of the boundary assembly given in Proposition~\ref{BA}.

\begin{teo}\label{blocks}
Fix $\varepsilon>0$ and $\eta\in(0,\eta_{1}]$.
\begin{enumerate}[(I)]
\item Let $A$ be an initial fundamental block of $\mc D$.
Then there exists a
holomorphic regular curve $\Upsilon_{A}\subset \mc T_{\eta_{1}}^{*}$
 transverse to $\F$ and
there exists a subset $\mc B_{\eta}(A)$ of $\mc T_{\eta}^{*}(A)$ satisfying
the following conditions:
\begin{enumerate}[(1)]
\item for $\eta'>0$ small enough the inclusion $\mc T_{\eta'}^{*}(A)\subset\mc B_{\eta}(A)$ is $\partial$-rigid;
\item $\mc B_{\eta}(A)$ is a $\F$-adapted block;
\item the connected components $\mc V_{1},\ldots,\mc V_{\mf n_{A}}$ of $\partial \mc B_{\eta}(A)$ are of suspension type over the connected components of $\partial A$;
\item $\mf c(\mc V_{j})\le \varepsilon$ for each $j=1,\ldots,\mf n_{A}$;
\item the intersection $\Upsilon_{A,\eta}:=\Upsilon_{A}\cap\mc B_{\eta}(A)$
is incompressible in $\mc B_{\eta}(A)$ and it
satisfies
 $\mr{Sat}_{\F}(\Upsilon_{A,\eta},\mc B_{\eta}(A))=\mc B_{\eta}(A)$ and $\Upsilon_{A,\eta}\loar\mc B_{\eta}(A)$.
\end{enumerate}
\item Let $A$ be a fundamental block of $\mc D$ which is not an initial or breaking block. Then there exists a
holomorphic regular curve $\Upsilon_{A}\subset \mc T_{\eta_{1}}^{*}$
 transverse to $\F$,
there exist a constant $C_{A}>0$ and a function $\rho_{A}:\R^{+}\to\R^{+}$ with $\lim\limits_{c\to 0}\rho_{A}(c)=0$, such that for every suspension type subset $\mc V\subset \mc T_{\eta}^{*}(A)$ over a connected component of $\partial A$ satisfying $\mf c(\mc V)\le C_{A}$, there exists a subset $\mc B_{\eta}(A)$ of $\mc T_{\eta}^{*}(A)$  satisfying Properties (1), (2), (3) and (5) of Part~(I) as well as
\begin{enumerate}[(3')]
\item[(3')] $\mc V_{1}\loar\mc V$;
\item[(4')] $\mf c(\mc V_{j})\le \rho_{A}(\mf c(\mc V))$ for each $j=1,\ldots,\mf n_{A}$.
\end{enumerate}
\item Let $A$ be a breaking block of $\mc D$.
Then for every choice of suspension type subsets $\mc V_{1},\ldots, \mc V_{\mf n_{A}}\subset\mc T_{\eta}^{*}(A)$ over the connected components of $\partial A$ such that the inclusion $\bigcup\limits_{i=1}^{\mf n_{A}}\mc V_{i}\subset\partial\mc T_{\eta}^{*}(A)$ be rigid, there exists a subset $\mc B_{\eta}(A)$ of $\mc T_{\eta}^{*}(A)$   satisfying Properties (1) and (2) of Part~(I), such that the connected components $\mc V_{1}',\ldots,\mc V_{\mf n_{A}}'$ of $\partial \mc B_{\eta}(A)$ are of suspension type and they satisfy
\begin{enumerate}[(3'')]
\item[(3'')] $\mc V_{j}'\loar\mc V_{j}$
\item[(4'')] $\mf c(\mc V_{j}')\le\varepsilon$
 \end{enumerate}
 for each $j=1,\ldots,\mf n_{A}$.
\end{enumerate}
\end{teo}

We will prove this theorem in the following section.

\section{Construction of foliated adapted blocks}\label{construction}
Theorem~\ref{blocks} is proved in \cite[Theorem 3.2.1]{MM} when the fundamental block $A$ is an aggregated block or a singularity block. Thus, it suffices to consider the cases of genus blocks and dicritical blocks which we treat separately in sections~\ref{genus} and~\ref{dicritical} respectively.

\subsection{Genus type foliated adapted block}\label{genus}

In the sequel we will assume that the genus of $D$ is $g>0$. In order to simplify the notations in this section we will denote
$$\Gamma:=\Gamma_{D},\qquad D_{\Gamma}:=D_{\Gamma_{D}},\qquad\textrm{and}\qquad D':=\overline{D\setminus D_{\Gamma_{D}}}.$$

\subsubsection{Preliminary constructions} We fix a \textbf{normal form} for $D$ given by
\begin{itemize}
\item a closed regular polygon $\mc P\subset \C$ of $4g$ sides of length $1$ centered at the origin;
\item arc-length parameterizations $a_{1},b_{1},a_{1}',b_{1}',\ldots,a_{g},b_{g},a_{g}',b_{g}'$ of the adjacent sides of $\mc P$ positively oriented according to $\partial \mc P$ such that $a_{1}(0)\in\R^{+}$;
\item a continuous map $\Psi:\mc P\to D$  such that the restriction of $\Psi$ to each side $|a_{j}|$ or $|b_{j}|$ is a smooth immersion and the compositions
$\alpha_{j}:=\Psi\circ a_{j}$, $\beta_{j}:=\Psi\circ b_{j}$, $j=1,\ldots,g$, are simple loops having a same origin $m_{\Lambda}$ which only meets each other in that point; moreover $\alpha_{j}^{-1}=\Psi\circ a_{j}'$ and $\beta_{j}^{-1}=\Psi\circ b'_{j}$ for $j=1,\ldots,g$;
\item $\Psi$ has an extension to an open neighborhood of $\mc P$ into $\C$ which is a local homeomorphism and its restriction to $\mc P\setminus\partial\mc P$ is a homeomorphism onto $D\setminus\Lambda$, where
$$\Lambda=\bigcup\limits_{j=1}^{g}|\alpha_{j}|\cup|\beta_{j}|$$ is a  wedge of $2g$ circles.
\end{itemize}

We fix an open disk $\mb D_{\epsilon}\subset\mc P$ centered at the origin of radius $\epsilon<\cos\left(\frac{\pi}{2g}\right)$. Up to modifying slightly $\Psi$ we can assume that $\Psi(\overline{\mb D_{\epsilon}})=\overline{D_{\Gamma}}$ so that the loop $\theta:[0,1]\to\Gamma=\partial D'$ given by $\theta(s)=\Psi(\epsilon e^{2i\pi s})$ is a simple parametrization of $\Gamma$. We consider the pull-back $\wh\pi:\wh W_{\mc P'}\to\mc P'$ by the restriction of $\Psi$ to $\mc P':=\mc P\setminus\mb D_{\epsilon}$ of the fibration $\pi_{D}:W_{D'}\to D'$, where $W_{D'}:=\pi_{D}^{-1}(D')\setminus D$. Thus, $\wh\pi$ is a  continuous $\mb D^{*}$-fibration which is globally trivial. Let $\wh \Psi:\wh W_{\mc P'}\to W_{D'}$ the continuous map which make commutative the cartesian diagram
$$
\begin{array}{ccc}
  \wh W_{\mc P'}&\stackrel{\wh\Psi}{\longrightarrow}& W_{D'} \\
  \wh \pi\downarrow\phantom{\wt \pi}&{\scriptstyle\square }&\downarrow\pi _{D} \\
  \mc P'&\stackrel{\Psi}{\longrightarrow}& D'\,.
\end{array}
$$
Clearly $\wh\Psi$ is a local homeomorphism whose restriction to $\wh \pi^{-1}(\mc P'\setminus\partial \mc P)$ is a homeomorphism onto $\pi_{D}^{-1}(D'\setminus\Lambda)\setminus D$.
The foliation $\F_{|W_{D'}}$ lifts to a regular foliation $\wh\F$ on $\wh W_{\mc P'}$ transverse to the fibres of $\wh\pi$.

We fix a conformal pointed disk $T\subset\pi_{D}^{-1}(m_{\Lambda})$ whose size and roughness is bounded by a constant $C_{D}>0$ small enough so that all the constructions we shall done in the sequel lead us to sets having finite size and roughness.
By construction there exists $\wh T\subset\wh \pi^{-1}(\wh m)$ such that $T=\wh\Psi(\wh T)$, where $\wh m$ is the vertex of $\mc P$ lying on $\R^{+}$. The image by $\wh\Psi$ of the suspension $\mb V_{\wh T,\wh\mu}$ of $\wh T$ via $\wh \pi$ over the loop $\wh \mu:=a_{1}\svee b_{1}\svee a_{1}'\svee b_{1}'\svee\cdots\svee a_{g}\svee b_{g}\svee a_{g}'\svee b_{g}'$ can be considered as the suspension of $T$ via $\pi_{D}$ over the loop
\begin{equation}\label{mu}
\mu:=\Psi\circ\wh\mu=\alpha_{1}\svee\beta_{1}\svee\alpha_{1}^{-1}\svee\beta_{1}^{-1}\svee\cdots\svee\alpha_{g}\svee\beta_{g}\svee\alpha_{g}^{-1}\svee\beta_{g}^{-1}.
\end{equation}
The subset $\underline{\mc B}:=\wh\Psi^{-1}(\wh\Psi(\mb V_{\wh T,\wh\mu}))$ of $\wh \pi ^{-1}(\partial \mc P')$ is not necessarily a multisuspension set in the sense of \cite[Definition 4.2.1]{MM} because over each vertex of $\mc P$ this set is the union of $4g$ pointed disks, two of them are contained in the adherence of $\underline{\mc B}\setminus\wh\pi^{-1}(\mc S_{\mc P})$, where $\mc S_{\mc P}$ is the vertex set of $\mc P$, but the other two could not satisfy this condition. We put
$$\wh{\mc B}_{\partial \mc P}:=\overline{\underline{\mc B}\setminus\wh\pi^{-1}(\mc S_{\mc P})},\qquad \mc B_{\Lambda}:=\wh\Psi(\wh{\mc B}_{\partial\mc P}).$$
The following diagram is commutative but not necessarily cartesian
$$
    \begin{array}{ccc}
  \wh {\mc B}_{\partial \mc P} &\stackrel{\wh\Psi}{\longrightarrow}&\mc B_\Lambda  \\
  \wh \pi\downarrow\phantom{\wt \pi}&{\scriptstyle\circlearrowleft }&\downarrow\pi _{D} \\
  \partial\mc P&\stackrel{\Psi}{\longrightarrow}& \Lambda \,.
\end{array}
$$
Now we will precise the geometry of $\mc B_{\Lambda}$. Let us denote by $h_{\alpha_{j}}$ (resp. $h_{\beta_{j}}$) the holonomy transformations of $\F$ along the loops $\alpha_{j}$ (resp. $\beta_{j}$), represented over the transverse section $\pi_{D}^{-1}(m_{\Lambda})$. If $C_{D}>0$ is small enough then the following pointed $4g$ disks are well defined
$$T_0 := T\,,\quad T _{4j+1}
=h_{\alpha _j}(T _{j})\,, \quad
 T _{4j+2} =
  h_{\beta _j}(T _{4j+1})\,,$$
   $$T _{4j+3} =h_{\alpha _j}^{-1}(T _{4j+2})\,,\quad
 T _{4j+4} =h_{\beta _j}^{-1}(T _{4j+3})\,,$$
$j=1,\ldots , g$. We have a decomposition $$\mc B_\Lambda = \mc
B_1\cup\mc B_2\cup\cdots\mc B_{2g}\,,$$ of $\mc B_\Lambda $ in  $2g$
pieces of suspension type
$$ \mc B_{2j-1} :=  \mb V_{T _{4j-4},\;\alpha _j}\cup \mb V_{T _{4j-2},\;\alpha
_j^{-1}} = \mb V_{T _{4j-4}\cup T _{4j-1},\; \alpha _j}\,,$$
$$
\mc B_{2j}:=\mb V_{T _{4j-3},\; \beta _j}\cup \mb V_{T _{4j-1},\;
\beta _j^{-1}} = \mb V_{T _{4j-3}\cup T _{4j},\; \beta _j}\, ,$$
$j=1,\ldots , g$,
with finite roughness. Moreover, $\pi_{D}(\mc B_{i}\cap\mc B_{j})=\{m_{\Lambda}\}$.
From this description follows:
\begin{enumerate}[($\ast$)]
\item\it If $\lambda:[0,1]\to L$ is a simple parametrization of a leaf $L$ of $\F_{|\mc B_{k}}$, $k=1,\ldots,2g$, then there exists a unique path $\wh\lambda:[0,1]\to\wh{\mc B}_{\partial\mc P}$ such that $\wh\Psi\circ\wh\lambda=\lambda$ and the orientations of $\wh\pi\circ\wh\lambda$ and $\partial\mc P$ coincide.
\end{enumerate}
\begin{lema}\label{sous}
If $\chi$ is a path lying on a leaf $L$ of $\mc B_{\Lambda}$ such that $\pi\circ\chi=\mu^{\nu}$ then there exists a unique path $\wh\chi:[0,1]\to\wh{\mc B}_{\partial\mc P}$ lying on a leaf of $\wh\F$ such that $\wh\Psi\circ\wh\chi=\chi$.
\end{lema}

\begin{dem}
We decompose $\chi=\chi_{1}\svee\cdots\svee\chi_{n}$ with $|\chi_{j}|\subset\mc B_{k_{j}}$. By property ($\ast$) each $\chi_{j}$ possesses a unique lift $\wh\chi_{j}$ with the same orientation as $\partial\mc P$. We must prove that all these lifts
glue in a unique continuous path $\wh\chi$. Fix $j\in\{1,\ldots,n\}$ and notice that the point $\chi_{j}(1)$ possesses exactly $4g$ pre-images by $\wh\Psi$, one  over each vertex of $\mc P$. To prove that $\wh\chi_{j}(1)=\wh\chi_{j+1}(0)$ it suffices to see that $\wh\pi\circ\wh\chi_{j}(1)=\wh\pi\circ\wh\chi_{j+1}(0)$. But $\wh\pi\circ\wh\chi_{j}$ is the unique lift of $\pi\circ\chi_{j}$ with the same orientation as $\partial\mc P$ and theses lifts glue because  $\pi \circ \chi $ lifts, by hypothesis.
\end{dem}

\begin{obs}\label{top}
Since $C_{D}>0$ is small enough so that the roughness of $\mc B_{j}$ is finite, we have that for every $\eta'>0$ small enough $\mc B_{\Lambda}$ retracts onto $\mc T_{\eta'}^{*}(\Lambda)$, which has the homotopy type of a product of a circle by the wedge of circles $\Lambda$. More precisely, for all $m_{0}\in\bigcup\limits_{k=1}^{4g}T_{k}$ the map
\begin{equation}\label{isopi1}
\chi : \pi _1(\mc B_\Lambda ,m_0) \rightarrow \pi _1(\Lambda ,
m_\Lambda )\oplus \Z\,,\quad [{\lambda}]_{\mc B_\Lambda } \mapsto
\left([\pi _D\circ \lambda ]_\Lambda ,\; \frac{1}{2i\pi
}\int_\lambda \frac{dz}{z}\right)\,,
\end{equation}
is an isomorphism.
\end{obs}

\begin{lema}\label{R}
There exist a neighborhood $\wh{\mc B}$ of $\mc P'$ in $\wh W_{D'}$ and two retractions by deformation
$r:D'\to\Lambda$ and $R:\mc{B}_{D'}:= \wh{\Psi }(\wh{\mc B})\to \mc B_{\Lambda}$
such that
\begin{enumerate}[(i)]
\item $\wh{\mc B}\cap \wh{\pi }^{-1}(\partial\mc P)=\wh{\mc B}_{\partial \mc P}$, $\mc B_{D'}\cap \pi _{D}^{-1}(\Lambda )=\mc B_{\Lambda }$ and the following  diagram is commutative:
\begin{equation*}
\begin{array}{ccc}
  \wh {\mc B}&\stackrel{\wh\Psi}{\longrightarrow}&\mc B_{D'} \\
  \wh \pi\downarrow\phantom{\wt \pi}&{\scriptstyle\circlearrowleft }&\downarrow\pi _{D} \\
  \mc P'&\stackrel{\Psi}{\longrightarrow}& D'\,.
\end{array}
\end{equation*}
\item $R$ and $r$ commute with the fibration $\pi_{D}$, i.e. $\pi_{D}\circ R=r\circ \pi_{D}$;
\item for every leaf $L$ of $\F_{|\mc B_{D'}}$, the restriction $R_{|L}$ is a retraction by deformation of $L$ onto $L\cap\mc B_{\Lambda}$;
\item every path $\gamma$ on $\mc B_{D'}$ with endpoints lying on the fibre $\pi_{D}^{-1}(m_{\Gamma})$ of a point $m_{\Gamma}\in\Gamma$ is homotopic inside $\mc B_{D'}$ to a path contained in $\mc B_{\Gamma} := \mc B_{D'}\cap \pi _{D}^{-1}(\Gamma )$ if and only if the element $[\pi_{D}\circ R\circ\gamma]_{\Lambda}$ of $\pi_{1}(\Lambda,m_{\Lambda})$ belongs to the subgroup generated by the loop $\mu$ defined in Equation~(\ref{mu});
\item a path $\gamma$ lying on a leaf $L$ of $\F_{\mc B_{D'}}$ with endpoints on $\pi_{D}^{-1}(m_{\Gamma})$ is homotopic inside $L$ to a path lying on {$\mc B_{\Gamma}\cap L$}
    if  $[\pi_{D}\circ R\circ\gamma]_{\Lambda}$ belongs to the subgroup $\langle\mu\rangle$ of $\pi_{1}(\Lambda,m_{\Lambda})$.
\end{enumerate}
\end{lema}
\begin{dem}
Let $\Phi(t,z)$ the flow of the radial vector field $\mc R=z\frac{\partial}{\partial z}$ on $\C$. If $z\in\mc P'$ we define $\zeta(z)=\inf\{t\in\R_{>0}\,|\,\Phi(t,z)\notin\mc P'\}$. The map
$$ \wh h :\mc P'  \times [0, 1]
\longrightarrow \mc P' \,, \qquad \wh h(z, t) := \Phi\left(t\, \varsigma(z),\, z\right)$$
is a homotopy defining a retraction by deformation $\wh r:= h(\cdot, 1) : \mc P' \rightarrow\partial \mc P$. Its restriction to $\partial \overline{\mb D}_\epsilon  \times [0, 1]$ is a homeomorphism sending each segment $\{z\}\times [0,1]$ onto the intersection of the half line $\R_{\geq 0} \cdot z$ with $\mc P'$. We define $r:=\wh\Psi\circ\wh r\circ (\Psi_{|\Gamma})^{-1}$. The vector field $\mc R$ lifts (via $\wh \pi$) to a unique vector field $\wh{\mc R}$ tangent to the foliation $\wh\F$. Let $\wh\Phi(t,m)$ be its flow and denote $\varsigma'(z) := \inf\{t\in \mb R_{>0}\;|\; \Phi(-t,z)\in
\mb D_\epsilon\}$. The map $(m,t)\mapsto \wh\Phi(-t\varsigma'(\wh\pi (m)) ,m)$
define a homeomorphism of $\wh{ \mc B}_{\partial \mc P'}\times [0,1]$ onto a neighborhood $\wh{\mc B}$ of $\mc P'$ in $\wh W_{D'}$. Consider now the homotopy
$$
\wh H : \wh{ \mc B}\times [0,1] \longrightarrow \wh{ \mc B}\,,\quad
\wh H(m, t) : = \wh\Phi(t\varsigma(\wh\pi (m)) ,m)
\,,\qquad \wh \pi \circ \wh H = \wh h\circ \wh \pi\,,$$
which lifts $\wt h$, and the homotopy
$$
\wh H ': \wh{ \mc B}\times [0,1] \longrightarrow \wh{ \mc B}\,,\quad
\wh H'(m, t) : = \wh\Phi(-t\varsigma'(\wh\pi (m)) ,m)
\,,\qquad \wh \pi \circ \wh H' = \wh h'\circ \wh \pi\,,$$ 
which lifts the homotopy
$$
\wh h' : \partial \overline{\mb D}_\epsilon  \times [0, 1]
\longrightarrow \mc P' \,, \qquad \wh h'(z, t) := \Phi\left(-t\, \varsigma'(z),\, z\right)\,.$$
Clearly the restrictions
$$
\wh H'_{|\wh{\mc B}_{\partial \mc P}\times[0,1]} : \wh{\mc
B}_{\partial \mc P}\times[0,1]\stackrel{\sim}{\longrightarrow}
\wh{\mc B}\ \ \hbox{and}\ \ \wh H_{|\wh{\mc B}_{\partial\mb D_{\epsilon}}\times[0,1]} : \wh{\mc B}_{\partial\mb D_{\epsilon}}\times[0,1]\stackrel{\sim}{\longrightarrow} \wh{\mc B}\,,
$$
are homeomorphisms which conjugate  the product foliations $\wh\F_{|\wh{\mc B}_{\partial \mc P}}\times [0,1]$ and
$\wh\F_{|\wh{\mc B}_{\partial\mb D_{\epsilon}}}\times [0,1]$ to the foliation  $\wh\F$, where $\wh{\mc B}_{\partial\mb D_{\epsilon}}=\wh\Psi^{-1}(\mc B_{\Gamma})\subset\wh\pi^{-1}(\partial\mb D_{\epsilon})$. The maps
$$
\wh R := \wh H(\cdot , 1) : \wh{\mc B}\longrightarrow \wh{\mc
B}_{\partial \mc P}\quad \hbox{and}\quad \wh R' := \wh H'(\cdot , 1)
: \wh{\mc B}\longrightarrow \wh{\mc B}_{\partial\mb D_{\epsilon}}
$$
are retractions by deformation inducing retractions by deformation on each leaf of $\wh\F$ lifting respectively $\wh r$ and
$$
    \wh r' :=\wh h'(\cdot, 1) : \mc P' \longrightarrow \partial \mb
D_\epsilon\,.
$$
Since the restriction of $\wh\Psi$ to $\wh{\mc B}\setminus\wh{\mc B}_{\partial\mb D_{\epsilon}}$ is a homeomorphism onto $\mc B_{D'}\setminus\mc B_{\Lambda}$, the map $\wh\Psi\circ\wh R\circ\wh\Psi^{-1}:\mc B_{D'}\setminus\mc B_{\Lambda}\to\mc B_{\Lambda}$ is well defined and it extends to a map $R:\mc B_{D'}\to\mc B_{\Lambda}$ by being the identity on $\mc B_{\Lambda}$. Indeed, the restriction of $\wh\Psi$ to each subset $\wh{\mc B}_k := \wh {\mc B}\cap \wh\pi
^{-1}\bigl(\bigl\{\frac{2\pi k}{4g} <\arg(z)<\frac{2\pi
(k+1)}{4g}\bigl\}\bigl)$ and $\wh{\mc B}'_k := \wh {\mc B}\cap \wh\pi_{D}
^{-1}\bigl(\bigl\{\arg(z)=\frac{2\pi k}{4g}\bigl\}\bigl)$,
$k=0,\ldots ,4g-1$, is a homeomorphism onto their image. Moreover $\wh R(\wh{\mc B}_k )= \wh{\mc B}_k \cap \wh {\mc
B}_{\partial \mc P}$, $\wh R(\wh{\mc B}'_k )= \wh{\mc B}'_k \cap \wh
{\mc B}_{\partial \mc P}$. Therefore the restriction of $\wh\Psi\circ \wh R\circ \wh \Psi^{-1}$ to each subset
$\wh \Psi(\wh{\mc B}_k)$, $\wh \Psi(\wh{\mc B}'_k)$ is well-defined and continuous. All these restrictions coincide with the identity map on $\mc B_{\Lambda}$ because $\wh R=\mr{Id}$ on $\wh\Psi^{-1}(\mc B_{\Lambda})=\wh{\mc B}_{\partial\mc P}$.
Thus $R:\mc B_{D'}\to\mc B_{\Lambda}$ is a retraction by deformation satisfying Properties (ii) and (iii) of the lemma. We shall see now that $R$ also satisfies Property (iv).

Let $\gamma:[0,1]\to\mc B_{D'}$, $\gamma(0),\gamma(1)\in\pi_{D}^{-1}(m_{\Gamma})$, $m_{\Gamma}\in\Gamma$, be a path homotopic to another path $\gamma_{1}$ lying on $\mc B_{\Gamma}$. It follows from (\ref{isopi1}) that we can take for $\pi_{D}\circ\gamma_{1}$ a power $\breve{\Gamma}^\nu $ of the simple parametrization $\breve{\Gamma }(t) :=
\Psi(\epsilon e^{2i\pi t})$, $t\in[0,1]$ of $\Gamma $ which satisfies  $r\circ \breve{\Gamma } = \mu $. Hence $\pi _D\circ R\circ
\gamma = r\circ \pi _D\circ \gamma = \mu ^\nu $.

Conversely, let $\gamma : [0,1] \rightarrow \mc
B_{D'}$, $\gamma (0)$, $\gamma (1)\in \pi _D^{-1}(m_\Gamma )$, be a path such that $\pi _D\circ R \circ \gamma = r\circ \pi _D\circ \gamma $  is homotopic to $\mu^{\nu}$. We consider a path $\xi$ contained in $\mc B_{\Gamma}$ having the same endpoints as $\gamma$ and such that $\pi_{D}\circ\xi\sim\breve{\Gamma}$. The loop $\delta:=\gamma\svee\xi^{-\nu}$ satisfy $[\pi _D\circ R\circ\delta]_\Lambda =0$. Consequently $R\circ\delta$ is homotopic in $\mc B_{\Lambda}$ to a loop lying on the fibre $\pi_{D}^{-1}(m_{\Lambda})$. Since $R$ is a retraction by deformation commuting to the projection $\pi_{D}$, we obtain that $\delta$ is homotopic inside $\mc B_{D'}$ to a loop $\delta_{1}$ lying on $\pi_{D}^{-1}(m_{\Gamma})$. Hence $\gamma$ is homotopic inside $\mc B_{D'}$ to the path $\delta_{1}\svee\xi^{-\nu}$ which is contained in $\mc B_{\Gamma}$.

Now, we shall prove Property (v) from the following assertion:
\begin{enumerate}[$(\star)$]
\item Let $\delta$ be a path lying on a leaf $L$ of $\F_{|\mc B_{\Lambda}}$ such that $\pi_{D}\circ\delta$ is homotopic to $\mu^{\nu}$. Then there exists a path $\chi$ homotopic to $\delta$ inside $L$ such that $\pi_{D}\circ\chi=\mu^{\nu}$.
\end{enumerate}

We can apply this property to the path $\delta:=R\circ\gamma$ because $\pi_{D}\circ R\circ\gamma$ is homotopic to $\mu^{\nu}$ for some $\nu\in\mb Z$ by hypothesis. Thus we obtain a path $\chi$ homotopic to $R\circ\gamma$ inside $L\cap \mc B_{\Lambda}$ such that $\pi_{D}\circ\chi=\mu^{\nu}$. By applying Lemma~\ref{sous} to it we get a continuous $\wh\Psi$-lift $\wh\chi$ lying on the leaf $\wh L=\wh\Psi^{-1}(L)$ of $\wh\F$. On the other hand, by using the foliated retraction $R$ we construct two paths $\xi_{0}$, $\xi_{1}:[0, 1]\to L$ such that   $\xi_{0}(0)=\gamma(0)$, $\xi_{0}(1)= R\circ \gamma (0)$, $\xi_{1}(0)= R\circ\gamma (1)$, $\xi_{1}(1)=\gamma(1)$ and
$$\gamma\sim_{L}\xi_{0}\svee(R\circ\gamma)\svee\xi_{1}\sim_{L}
\xi_{0}\svee\chi\svee\xi_{1}=\wh\Psi
\circ(\wh\xi_{0}\svee\wh\chi\svee\wh\xi_{1})$$
for the unique continuous $\wh\Psi$-lifts $\wh\xi_{0}$, $\wh\xi_{1}$ of
$\xi_{0}$ and $\xi_{1}$ respectively. Moreover $|\xi_0|$ and $|\xi_1|$ are contained in orbits of $\wh{\mc R}$ and clearly $\wh\xi_{0}\svee\wh\chi\svee\wh\xi_{1}$ is homotopic in $\wh L$ to $\wh R'\circ \wh\chi$. Then $\gamma $ is homotopic in $L$ to $\wh \Psi\circ \wh\chi$ which is contained  in $\mc B_{\Gamma }\cap L$.\\

In order to prove Assertion $(\star)$ we consider a path $\delta:[0,1]\to L$ satisfying the hypothesis of this assertion. Without loss of generality we can assume that $\delta$ is smooth and transverse to the fibre $\pi_{D}^{-1}(m_{\Lambda})$. We get a subdivision  $t_1=0<t_2<\cdots <t_{q'+1}=1$ of the interval  $[0,1] $, such that each curve $\delta([t_{j},t_{j+1}])$ is contained in a single block $\mc B_{\tau(j)}$. The endpoints $m_j := \delta (t_j)$
and $m_{j+1} := \delta (t_{j+1})$ of the path $\delta _j := \delta
_{|[t_j, t_{j+1}]}$ project by $\pi_{D}$ onto the point $m_{\Lambda}$ and the image of $\delta_{j}$ is the closed segment $L_{j}$ of the leaf $L$ (of real dimension one) delimited by the points $m_{j}$ and $m_{j+1}$. The projection of the path $\delta_{j}$ by $\pi_{D}$ is a loop based on $m_{\Lambda}$ with image $|\alpha_{\tau(j)}|$ or $|\beta_{\tau(j)}|$ depending on the parity of $\tau(j)$. Moreover, the equality $\pi_{D}\circ\delta_{j}(t)=m_{\Lambda}$ only holds for $t=t_{j}$ and $t=t_{j+1}$. Thus, if we assume that $\pi_{D}\circ\delta_{j}$ is not null-homotopic then there exists a homotopy inside $|\pi_{D}\circ\delta_{j}|$ between $\pi_{D}\circ\delta_{j}$ and one of the loops $\alpha _{\tau
(j)}$, $\alpha _{\tau (j)}^{-1}$, $\beta _{\tau (j)}$ or $\beta
_{\tau (j)}^{-1}$. Clearly this homotopy lifts to a homotopy inside $L_{j}$ between $\delta_{j}$ and a new path $\rho_{j}$. Finally, there exists $q\le q'$ such that $\delta$ is homotopic inside $L$ to the path
$$\rho := \rho _1\svee \cdots \svee \rho _{q}\,,\quad
\pi _D\circ \rho _j =: \mu _{\tau (j)} \in \{\, \alpha _{\tau
(j)}\,,\alpha _{\tau (j)}^{-1}\,,\beta _{\tau (j)}\,,\beta _{\tau
(j)}^{-1} \,\}\,.$$
Let us consider the word
$$ M(\delta ) := \mu _{\tau
(1)}\,\mu _{\tau (2)}\cdots \mu _{\tau (q)}$$
composed by the signs of the alphabet
$$\mc A:= \{\alpha _1, \alpha _1^{-1},\ldots,\alpha _g,
\alpha _g^{-1}, \beta _1, \beta _1^{-1},\ldots,\beta _g, \beta
_g^{-1}\}\,.$$
By a sequence of moves of type
\begin{equation}\label{suppression}
    u_1\cdots u_k \,v\, v^{-1}\,u_{k+1} \cdots u_N\;\rightarrow
\;u_1\cdots u_k\,u_{k+1} \cdots u_N\,,\qquad u_j,\; v \in \mc A\,,
\end{equation}
$$
    u_1\cdots u_k \,v\, v^{-1}\,u_{k+1} \cdots u_N\;\leftarrow
\;u_1\cdots u_k\,u_{k+1} \cdots u_N\,,\qquad u_j,\; v \in \mc A\,,
$$
we can transform the word $M(\delta)$ in a unique word $M(\delta)^{\mr{red}}$, called the \textbf{reduced word} associated to $M(\delta)$, which do not contain any sub-word of length two of type $v\,v^{-1}$, $v\in\mc A$. The uniqueness of $M(\delta)^{\mr{red}}$ follows from the solution of the word problem in a free group, cf. \cite{Magnus}. More precisely, every element of the free group
$$
\pi _1(\Lambda , m_\Lambda ) \,\iso\, \langle\dot{\alpha} _1, \ldots
\dot{\alpha} _g, \dot{\beta} _1, \ldots \dot{\beta} _g\;| \, - \rangle\,,
$$
can be written in a unique way in the form  $\dot{M} :=
\dot{u}_1\svee\cdots \svee \dot{u}_p$, where $M:= u_1\cdots u_p$,
$u_j\in \mc A$ is a reduced word. Now we use the hypothesis that $\pi _D\circ \delta $  is homotopic inside $\Lambda$ to a loop of type $\mu^{\nu}$. We have equality of reduced words:
$$ M(\delta
)^\mathrm{red} = \overline{\mu }^{\nu}\,,\quad \overline{\mu } := \alpha
_1\beta _1\alpha _1^{-1}
   \beta _1^{-1}\cdots \alpha _g\beta _g\alpha _g^{-1}\beta
   _g^{-1}\,.
$$
Notice that we pass from $M(\delta )$ to $M(\delta )^\mathrm{red}$ by a sequence of suppression moves of type (\ref{suppression})
$$
M_0 = M(\delta ) \rightarrow \cdots \rightarrow M_q = M(\delta
)^\mathrm{red}\,, \quad M_k = u_{k,\, 1}\cdots u_{k,\, n_k}\,,\quad
u_{k,\,j}\in \mc A\,.
$$
To finish the proof it suffices to remark that
\begin{enumerate}[-]
  \item to any suppression move we can associate a homotopy inside $\Lambda$ between the loops $\underline{M}_k := u_{k,\,
1}\svee\cdots\svee u_{k,\, n_k}$ and $\underline{M}_{k+1} :=
u_{k+1,\, 1}\svee\cdots\svee u_{k+1,\, n_{k+1}}$,
\item if one of the paths  $\rho _j$ composing  $\rho $ satisfies  $\mu _{\tau (j+1)}=
  \mu _{\tau (j+1)}^{-1}$, then there exists a homotopy inside $L$ between
 $\rho _j\svee\rho _{j+1}\svee \rho _{j+2}$ and a path
  $\wt \rho _{j+2}$ such that $\pi _D\circ \wt \rho _{j+2}= \mu _{\tau
  (j)}$.
\end{enumerate}
\end{dem}

\subsubsection{Checking the properties of Theorem~\ref{blocks}}
First of all, we will precise the construction of the foliated block $\mc B_{\eta}(A)$ associated to the genus fundamental block $A=D'$. In the construction made in the precedent section we take a conformal disk $T\subset\pi_{D}^{-1}(m_{\Lambda})$ of size small enough so that it is contained in the open $4$-Milnor tube $\mc T_{\eta}$. Since $\mc B_{\Gamma}$ is a subset of type multi-suspension we can apply to it the \emph{rabotage} procedure described in \cite[Definition 4.3.5]{MM} in order to obtain a subset $\mc R_{\Gamma}$ of $\mc B_{\Gamma}$ of suspension type such that the inclusion $\mc R_{\Gamma}\subset\mc B_{\Gamma}$ is rigid and verifies $\mc R_{\Gamma}\uco{\F}\mc B_{\Gamma}$. Then we define
$$\mc B_{\eta}(A):=(\mc B_{D'}\setminus\mc B_{\Gamma})\cup\mc R_{\Gamma}.$$
Clearly the inclusion  $\mc B_{\eta}(A)\subset\mc B_{D'}$ is $\partial$-rigid.\\

We begin checking the properties of the part (I) in Theorem~\ref{blocks}. In order to see the point (1) we consider the following commutative diagram induced by the natural inclusions:
$$\begin{array}{ccc}
\pi_{1}(\mc T_{\eta'}^*(\Lambda))& \longrightarrow & \pi_{1}(\mc T_{\eta'}(D'))\\
\downarrow & & \downarrow\\
\pi_{1}(\mc B_{\Lambda}) & \longrightarrow & \pi_{1}(\mc B_{D'})
\end{array}$$
Remark~\ref{top} implies that the first vertical arrow is an isomorphism. The bottom horizontal arrow is also an isomorphism because the map $R$ in Lemma~\ref{R} is a deformation retract. By lifting conveniently the retraction $r$ to $\mc T_{\eta'}(D')$ we see that the top horizontal arrow is also an isomorphism. The fourth arrow is also an isomorphism. Consequently the inclusion $\mc T_{\eta'}(D')\subset\mc B_{D'}$ is rigid. The inclusion  $\mc T_{\eta'}(D')\subset\mc B_{\eta}(A)$ is also rigid. The fact that it is also $\partial$-rigid follows immediately from the construction.

In order to show (2) we must prove Properties (B1)-(B4) of Definition~\ref{block}:
\begin{enumerate}[(B1)]
\item $\partial\mc B_{\eta}(A)$ is incompressible in $\mc B_{\eta}(A)$ because the inclusion  $\partial\mc T_{\eta'}^*(D')\subset \partial\mc B_{\eta}(A)$ is rigid and  $\partial\mc T_{\eta'}^*(D')$ is incompressible in $\mc T_{\eta'}^*(D')$ thanks to Proposition~\ref{BA}.
\item The boundary $\partial \mc B_{\eta}(A)$ is transverse to $\F$ because it is  of  suspension type.
\item Since $\partial\mc B_{\eta}(A)$ has been obtained by the \emph{rabotage} procedure from a multi-suspension type subset $\mc B_{\Gamma}$, in order to prove the $1$-$\F$-connected\-ness of $\partial\mc B_{\eta}(A)$ inside $\mc B_{\eta}(A)$ it suffices to show that  $\mc B_{\Gamma}\uco{\F}\mc B_{D'}$ because $\partial\mc B_{\eta}(A)\uco{\F}\mc B_{\Gamma}$. In order to prove this, we consider a leaf $L$ of $\mc B_{D'}$ and two paths $a:[0,1]\to \mc B_{\Gamma}$ and $b:[0,1]\to L$ which are homotopic in $\mc B_{D'}$ By point (3) of Lemma~\ref{R} we deduce that $[\pi_{D}\circ R\circ b]\in\langle\mu\rangle\subset\pi_{1}(\Lambda,m_{\Lambda})$. By applying point (4) of Lemma~\ref{R} we obtain a new path $c:[0,1]\to L\cap\mc B_{\Gamma}$ which is homotopic to $b$ inside $L$. By transitivity, $c$ is homotopic to $a$ in $\mc B_{D'}$. Since $|a|$ and $|b|$ are contained in $\mc B_{\Gamma}$ which is incompressible in $\mc B_{D'}$ we conclude that $a$ is homotopic to $c$ in $\mc B_{\Gamma}$.
\item After point (2) of Lemma~\ref{R} we know that every leaf $L$ of $\mc B_{\eta}(A)$ is a deformation retract of $L\cap\mc B_{\Lambda}$, which outside of the fibre $\pi_{D}^{-1}(m_{\Lambda})$ is a suspension type subset. We deduce that every leaf $L\cap\mc B_{\Lambda}$ of $\F_{|\mc B_{\Lambda}}$ is incompressible.
\end{enumerate}
Properties (3) and (4) of Part (I) are trivial because in this case $\mf n_{A}=1$. To see (5) we define first $\Upsilon_{D'}$ as the holonomic transport of $\pi_{D}^{-1}(m_{\Lambda})\cap \mc T_{\eta_{1}}^*(D')$ along the oriented segment joining the point $m_{\Lambda}$  to $m_{\Gamma}$. It is clear that $\Upsilon_{D'}\cap\mc B_{\eta}(A)$  is incompressible in $\mc B_{\eta}(A)$ and that  $Sat_{\F_{|\mc B_{\eta}(A)}}(\Upsilon_{D'},\mc B_{\eta}(A))=\mc B_{\eta}(A)$. On the other hand,  $\Upsilon_{D'}\cap\mc B_{\eta}(A)\uco{\F}\partial \mc B_{\eta}(A)\uco{\F}\mc B_{\eta}(A)$ because $\partial\mc B_{\eta}(A)$ is of suspension type.\\

To prove Part (II) of Theorem~\ref{blocks} we recall that if $A=D'$ is not an initial block then the holonomy transformation $h_{\Gamma}$ associated to $\Gamma$ is linearizable. Therefore, there exists a conformal disk $\Sigma\subset\pi_{D}^{-1}(m_{\Gamma})$ such that $h_{\Gamma}(\Sigma)\subset\Sigma$ or $h_{\Gamma}^{-1}(\Sigma)\subset\Sigma$. We define $\mc V_{1}=\mb V_{\Sigma,\Gamma}$   and we  begin the precedent construction with the conformal disk $T\subset\pi_{D}^{-1}(m_{\Lambda})$ obtained by holonomic transport of $\mc V_{1}\cap\pi_{D}^{-1}(m_{\Gamma})$
 along the segment joining $m_{\Gamma}$ to $m_{\Lambda}$, choosing $m_{\Lambda}$ as the breaking point of $\mc V_{1}$.
Indeed, from this choice the precedent construction shows that   $\mc V_{1}$ is of suspension type and   $\partial\mc B_{\eta}(A)=\mc V_{1}\loar\mc V$. Thus, we have proved point (3'). Since
$\mf n_{A}=1$ we can take $\rho_{A}(c)=c$ to obtain trivially (4'). \\

Finally, by definition a genus block is not a breaking block, so Part (III) do not apply in this case.

\subsection{Dicritical type foliated adapted block}\label{dicritical}
We fix a fundamental block $A\subset\mc D$ associated to a dicritical
 irreducible component $D$ of $\F$ of genus $g$ and valence $\mf n_{A}\ge 1$, given by Definition~\ref{defblocks}.
Condition (c) in Definition~\ref{adapted} implies that there are no dead branches adjacent to $D$.

Each connected component of $\partial A$ is the boundary of a closed disk $D_{s_{i}}^{(i)}$ contained in an adjacent component $D^{(i)}$ of $\mc D$ and $D\cap D^{(i)}=\{s_{i}\}$, $i=1,\ldots,\mf n_{A}$.
Let $\mc V_{i}$ be the given suspension sets over $\partial D^{(i)}_{s_{i}}$.
Since the holonomy of $\partial D^{(i)}_{s_{i}}$ is the identity we can choose a saturated subset $\mc V_{i}'\subset\mc V_{i}$ having $\mf c(\mc V_{i}')\le \varepsilon$, i.e. satisfying Condition (4''). The saturation condition of $\mc V_{i}'$ inside $\mc V_{i}$ implies that each $\mc V_{i}'$ is of suspension type and satisfies Property~(3'').

Next we define  $\mc B_{\mc V_{i}'}$ as the saturation of $\mc V_{i}'\subset\mc V_{i}$ by $\F$ inside $\pi_{D^{(i)}}^{-1}({D}_{s_{i}}^{(i)})$, where $\pi_{D^{(i)}}$ is the Hopf fibration over the component $D_{i}$. We put $\mc B_{\mc V'}:=\bigcup\limits_{i=1}^{\mf n_{A}}\mc B_{\mc V_{i}'}$ and we finally define
$$
\mc B_{\eta}(A):=\left(\pi_{D}^{-1}\left(D\setminus\mc B_{\mc V'}\right)\cap \mc T_{\eta}^*(A)\right)\cup\left(\mc B_{\mc V'}\setminus D\right).$$
Recall that we have choose Hopf fibration $\pi_{D}$ to be constant along the leaves of $\F|_{W_{D}}$, see Point (c) of Remark~\ref{obs14}.

In order to prove Part (III) of Theorem~\ref{blocks} it suffices to show Assertions~(1) and (2) of Part (I) because
 $\partial \mc B_{\eta}(A)=\bigcup\limits_{i=1}^{\mf n_{A}}\mc V_{i}'$ is automatically satisfied by construction.
It is clear that $\overline{\mc B}_{\eta}(A)$ is a tubular neighborhood of $A$ so that $\mc B_{\eta}(A)$ contains $\mc T_{\eta'}^{*}(A)$ for $\eta'>0$ small enough. This inclusion is $\partial$-rigid because the inclusions $\partial\mc B_{\eta}(A)\subset\partial\mc T_{\eta}^{*}(A)$ and $\partial\mc T_{\eta'}^{*}(A)\subset\partial\mc T_{\eta}^{*}(A)$ are rigid and, on the other hand, we can easily see that $\mc T_{\eta'}(A)\subset \mc B_{\eta}(A)$ is a retract by deformation and consequently this last inclusion is also rigid.

To prove that $\mc B_{\eta}(A)$ is a $\F$-adapted block it suffices to observe the following assertions  concerning properties (B1)-(B4) of Definition~\ref{block}:
\begin{enumerate}[(B1)]
\item By using Proposition~\ref{BA}, $$\pi_{1}(\partial\mc B_{\eta}(A))\cong\pi_{1}(\partial\mc T_{\eta}^*(A))\hookrightarrow\pi_{1}(\mc T_{\eta}^*(A))\cong\pi_{1}(\mc B_{\eta}(A))$$ after the $\partial$-rigidity of $\mc B_{\eta}(A)\subset\mc T_{\eta}^*(A) $
that we have seen before.
\item The boundary $\partial \mc B_{\eta}(A)=\bigcup\limits_{i=1}^{\mf n_{A}}\mc V_{i}'$ is a suspension type subset over  $\partial A$ and consequently it is  transverse to $\F$.
\item Each connected component of
$\bigcup\limits_{L\in\F}(L\cap\partial\mc B_{\eta}(A))$ is
diffeomorphic to the product $\mb D^*\times\mb D^*$ endowed with the horizontal foliation. Consequently, we have that  $\partial\mc B_{\eta}(A)\uco{\F}\mc B_{\eta}(A)$.
\item Every leaf $L$ of $\F_{|\mc B_{\eta}(A)}$ is diffeomorphic to $\mb D^{*}$ and a generator of $\pi_{1}(L)$ is sent to the element $c$ contained in the center of the group $\pi_{1}(\mc B_{\eta}(A))$ which is isomorphic to the direct sum of $\mb Z c$ and a free group of rank
$2g+\mf n_{A}-1$.
\end{enumerate}

\section{Proofs of the main results}\label{sec4}

\subsection{Proof of Theorem~\ref{main}}
 Recall that the break graph associated to $(\F,\mc D)$ was obtained by considering the complement of the breaking elements $\mc R$ inside $\mc G_{\mc D}$, see Introduction.
We consider the graph $\check{\mc G}$ obtained by eliminating
the part of the break graph associated to $(\F,\mc D)$ corresponding to the dead branches of $\mc D$.
After Condition~(e) in Definition~\ref{adapted} and Hypothesis (G) on $(\F,\mc D)$, each connected component $\Lambda$ of $\check{\mc G}$ is a tree with at most one vertex corresponding to an initial component  $C\subset\mc C$.
 We apply Part (I) of Theorem~\ref{blocks} to the initial block $A_{C}$ associated to $C$.
If $\Lambda$ does not contain any initial element then we begin the construction from a fundamental block  $A$ associated to an  arbitrary  element of $\Lambda$ by applying Part (II). To do that we choose some suspension type initial boundary $\mc V$ with $\mf c(\mc V)$ small enough.
Since the fundamental blocks $A\neq A_{C}$ corresponding to elements of $\Lambda$ are not initial blocks we can apply to them by adjacency order Part (II) of Theorem~\ref{blocks} from the suspension type boundary obtained in the precedent step.
Since $\Lambda$ is finite this procedure stops. In this way we obtain a $\F$-adapted block for each fundamental block of $\mc D$
except for the breaking blocks of  $\mc D$.
The size and roughness of the boundary of the $\F$-adapted block obtained at each step of this inductive process is controlled by those of the block constructed in the precedent step.
If we choose the size and roughness  sufficiently small  at the beginning then we have finite roughness at each step of the induction, see \cite[\S3.2]{MM} for more details.
We make the boundary assembly of these $\F$-adapted blocks obtaining a connected subset $\mc B_{\eta}(\Lambda)$ of $\mc T_{\eta}^*(\Lambda)$ for each connected component $\Lambda$ of $\check{\mc G}$.

In order to make the boundary assembly of all these sets  $\mc B_{\eta}(\Lambda)$ we need also to consider $\F$-adapted blocks associated to the breaking elements $\rho\in\mc R$ adjacent to two connected components  $\Lambda$ and $\Lambda'$ of {$\check{\mc G}$},  which we construct from the suspension type boundaries of  $\mc B_{\eta}(\Lambda)$ and $\mc B_{\eta}(\Lambda')$
by using Part (III) of Theorem~\ref{blocks}. Notice that the case  $\Lambda=\Lambda'$ is not excluded. In fact, this situation could happen when $\partial \mc B_{\eta}(\Lambda)$ is not connected.

In this way we obtain a foliated boundary assembly
$$\mc B_{\eta}=\bigcup\limits_{\Lambda\subset{\check{\mc G}}}\mc B_{\eta}(\Lambda)\cup\bigcup\limits_{\rho\in\mc R}\mc B_{\eta}(\rho)\subset\mc T_{\eta}^*.$$
We take  $U_{1}=E(\mc B_{\eta_{1}})\cup\cD$.
There exists $\eta_{2}>0$ such that $\mc T_{\eta_{2}}^*\subset\mc B_{\eta_{1}}$ and we define  $U_{2}=E(\mc B_{\eta_{2}})\cup\cD\subset U_{1}$.
By induction we construct a decreasing a sequence  $(\eta_{n})$
tending to zero such that  $U_{n}:=E(\mc B_{\eta_{n}})\cup\cD\subset U_{n-1}$. Put
 $\Upsilon:=\sqcup_{A}\Upsilon_{A}$, where $A$ varies in the set of fundamental blocks of  $\mc D$ which are not breaking blocks.
To finish it suffices to remark the validity of the following assertions:
\begin{enumerate}[(i)]
\item The inclusion $U_{n+1}^*\subset U_{n}^*$ is  rigid by Remark~\ref{rigid}, Corollary~\ref{rig2} and Property  (1) of Theorem~\ref{blocks}.
\item Every leaf  $L$ of $\underline{\F}_{|U_{n}^*}$ is incompressible after Property (2) of Theorem~\ref{blocks}  by using Localization Theorem~\ref{loc}.
\item Thanks to Property (5) of Theorem~\ref{blocks} each irreducible component of $Y_{n}^{*}$ is incompressible in the corresponding $\F$-adapted block, which is incompressible in  $U_{n}^{*}$ by Localization Theorem~\ref{loc}. Hence
$Y_{n}^{*}$ is incompressible in $U_{n}^{*}$.
Let $\Omega$ be the union of all the $\F$-adapted blocks associated to non-breaking fundamental blocks of $\mc D$. Thanks to Property (5) in Theorem~\ref{blocks} we have $\mr{Sat}_{\F}(\Upsilon\cap \mc B_{\eta},\Omega)=\Omega$. Clearly, the connected components of $\partial\Omega$ are exactly the connected components of the boundary of all $\F$-adapted blocks associated to  breaking fundamental blocks of $\mc D$. Finally,  for each $\F$-adapted block $B$ associated to a fundamental breaking block $A$ of $\mc D$ we have
that $B\setminus\mr{Sat}_{\F}(\partial B,B)$ is a nodal or dicritical separator according to whether $A$ is a dicritical block or a singularity block (necessarily associated to a nodal singularity).
\item Property (iv) of Theorem~\ref{main} is equivalent to the relation $Y_{n}^{*}\underset{\scriptscriptstyle\uF}\looparrowright U_{n}^{*}$. This  follows from
 $\sqcup_{A}\Upsilon_{A}\uco{\F}\mc B_{\eta}$ because
$\Upsilon_{A}\uco{\F}\mc B_{\eta}(A)$ by Theorem~\ref{blocks},
$\mc B_{\eta}(A)\uco{\F}\mc B_{\eta}$ by Localization
Theorem~\ref{loc} and the transitivity of the relation $\uco{\F}$.
\item Let $U_{n}$ be one of the open sets that we have constructed. We still denote by $\widetilde{\F}_{U_n}$ the the pull-back
  by the universal covering $q_{U_n}:\widetilde{U_n}\to U_n^*$ of the foliation $\uF$ restricted to $U_n^\ast$ and we denote  $\wt{\mc Q}_{U_n}$ its leaf space.
It is easy to see that the open subset of $\wt{\mc Q}_{U_n}$ corresponding to leaves of $\wt\F_{U_n}$ projecting onto an open fixed separator has a natural structure of Hausdorff one-dimensional complex manifold. To obtain a complete holomorphic atlas on $\wt{\mc Q}_{U_n}$ we proceed as follows.
From the fact that  $Y_n^\ast$ is incompressible and  $1$-$\uF$-connected in $U_n^*$
follows that each connected component  $\wt Y_{\alpha}\cong\mb D$ of $q_{U_n}^{-1}(Y_n^\ast)$ intersects every leaf of $\wt\F_{U_n}$ in at most one point.
Consequently, the open canonical maps $\tau_{\alpha}:\wt Y_{\alpha}\to\wt{Q}_{U_n}$, sending each point $p\in\wt  Y_{\alpha}$ to the leaf $L_{p}$ of $\wt\F_{U_n}$ passing through  $p$, are injective.
Hence the inverse maps $\tau_{\alpha}^{-1}$ are holomorphic charts on $\wt{\mc Q}_{U_n}$.
We achieve the proof by noting that $U_n\setminus Sat_{\underline{\F}}(Y_n^\ast,U_n^*)$ is a disjoint finite union of nodal and dicritical separators and that the transition functions induce the holonomy pseudo-group of $\uF$; hence they are holomorphic.
\end{enumerate}

\subsection{Proof of Corollary~\ref{B}}
We must check that the total transform of $\mc Z$ by the minimal reduction map $\pi$ of $\F$ is an adapted divisor of $(\pi^{*}\F,\pi^{-1}(0))$. Conditions~(a) and~(b) of Definition~\ref{adapted} are obviously fulfilled. Condition~(d) can not occur by the existence of local separatrices.

To prove Condition~(c) notice that on a neighborhood of a dead branch with branching point lying on a dicritical component all the leaves are compact. This situation can not happen because it does not exist compact analytic curves in $\C^{2}$.

To prove Condition~(e) we will use the well-known fact that the total divisor of the desingularisation of a germ of curve $(X,0)$ contains at most one irreducible component of the exceptional divisor adjacent to at least two dead branches.
We take for $X$ the union of the isolated separatrices of $\F$ and two non-isolated separatrices for each dicritical component of $\pi^{*}\F$.
We can easily check that the minimal desingularisation morphism of $X$ coincide with $\pi$, see \cite[Theorem~2]{CSL}. Then there exists at most one initial component of $(\pi^{*}\F,\pi^{-1}(0))$.

\subsection{Proof of Theorem~\ref{main2}}
As we have already point out  in the introduction, the equivalences $(1)\Leftrightarrow(1')$ and $(2)\Leftrightarrow(2')$ follow from  the main result of \cite{Rebelo} thanks to Condition~(R).
Since the implication $(2)\Rightarrow(1)$ is obvious it only remains to prove  implications $(1)\Rightarrow(3)\Rightarrow(4)\Rightarrow(2)$. To do that we will make a strong use of the notions and statements introduced in \cite{monodromy}.\\

\noindent $\mathbf{(4)\Rightarrow(2)}$\textbf{:}
Conditions~(a) and (b) in~(4) imply that if $D\subset\mc D$ is a dicritical component of $(\F,\cD)$ then $\varphi(D)$ is a dicritical component of $(\F',\cD')$ and if $s\in\cD$ is a nodal singularity of $\F$ then $\varphi(s)$ is a nodal singularity of $\F'$. Consequently $\varphi$ sends connected components of the cut divisor $\cD^{\mr{cut}}$ defined in the introduction into connected components of $\cD'^{\mr{cut}}$.
On the other hand, by assumption $(\psi,\wt\psi,h):=(\varphi_{|\Sigma},\wt\varphi_{|\wt\Sigma},h)$ is a realization of the $\mc S$-conjugation $(\varphi,\wt\varphi,h)$ between the monodromies $\mf m^{\uF}$ and $\mf m^{\uF'}$ over $\mc S$-collections of transversals $(\Sigma,\Sigma')$ in the sense of \cite[Definition~3.6.1]{monodromy}. Moreover, by definition it satisfies trivially the additional condition
\begin{equation*}
\wt\psi_{\bullet}=\wt\varphi_{\bullet}:\pi_{0}(\wt\Sigma)\to\pi_{0}(\wt\Sigma')
\end{equation*}
required in the Extension Lemma of \cite[Lemma~8.3.2]{monodromy}, whose proof is also valid for genus blocks. If $\cD^{\mr{cut}}$ is not a tree we choose singularity blocks $B_{\alpha}$ such that $\cD^{\mr{tree}}:=\cD^{\mr{cut}}\setminus\bigcup B_{\alpha}$ does not contain any cycle of components.
Let $B_{\alpha}'$ be the singularity block of $\cD'^{\mr{cut}}$ corresponding to $B_{\alpha}$ by $\varphi$ and put $\cD'^{\mr{tree}}:=\cD'^{\mr{cut}}\setminus\bigcup B_{\alpha}'$.
By applying iteratively the Extension Lemma beginning by $(\psi,\wt\psi,h)$ we obtain a realization $(\psi^{0},\wt\psi^{0},h)$  of $(\varphi,\wt\varphi,h)$ over a union $W$ of foliated adapted blocks covering $\cD^{\mr{tree}}$ and $\cD'^{\mr{tree}}$.

Now we fix
 transversal disks  $\Upsilon,\Theta$  to the local separatrices associated to the singularity block $B_{\alpha}$
contained in the boundary of $W$.
 Extension Lemma implies that $\psi^{0}$ is excellent and that $\psi^{0}$ and $\varphi$ coincide over $\cD$. Consequently,
$\Upsilon':=\psi^{0}(\Upsilon)=\varphi(\Upsilon)$ and
$\Theta':=\psi^{0}(\Theta)=\varphi(\Theta)$ are transversal disks to the local separatrices associated to the singularity block $B_{\alpha}'$.
Let $(\psi^{1},\wt\psi^{1},h)$ be the restriction of the realization $(\psi^{0},\wt\psi^{0},h)$ to $\Upsilon$.

Applying again Extension Lemma to the realization  $(\psi^{0}_{|\Theta},\wt\psi^{0}_{|\wt\Theta},h)$ for the block $B_{\alpha}$ we obtain a new realization
whose restriction $(\psi^{2},\wt\psi^{2},h)$ to $\Upsilon$ satisfies $\psi^{2}(\Upsilon)=\Upsilon'$,
 $$\wt\psi^{1}_{\bullet}=\wt\varphi_{\bullet}=\wt\psi^{2}_{\bullet}:\pi_{0}(\wt\Upsilon)\to \pi_{0}(\wt\Upsilon'),$$ and
the commutativity of the following diagrams
$$\begin{array}{ccc}
\wt\Upsilon_{\alpha} & \hookrightarrow & \mc Q^{\uF}\\
\wt\psi^{i}\downarrow\hphantom{\wt\psi^{i}} & & \hphantom{h}\downarrow h\\
\wt\Upsilon'_{\wt\varphi_{\bullet}(\alpha)} & \hookrightarrow & \mc Q^{\uF'}
\end{array}$$
for all $\alpha\in\pi_{0}(\wt\Upsilon)$ and $i=1,2$.
Since the horizontal arrows of these diagrams are monomorphisms we deduce that $(\psi^{1},\wt\psi^{1},h)=(\psi^{2},\wt\psi^{2},h)$. Consequently, we can glue these realizations to obtain a new realization $(\Psi,\wt\Psi,h)$ defined in a union of adapted foliated blocks covering $\cD^{\mr{cut}}$.

Finally, it only remains to extend $\Psi$ to the dicritical components and the nodal singularities in order to obtain a global realization of $(\varphi,\wt\varphi,h)$ which will be the desired $\mc S$-transversely holomorphic conjugation between $(\F,\cD)$ and $(\F',\cD')$. In fact, the extension to nodal singularities has been described in \cite[\S8.5]{monodromy}.

Now we fix  dicritical  components $D\subset\cD$ and $D':=\varphi(D)\subset\cD'$.
On neighborhoods of  these components the foliations $\F$ and $\F'$ are disk fibrations.
Because $D$ and $D'$ have the same self-intersection number,
we can identify two tubular neighborhoods of $D$ and $D'$ endowed with the restriction of the foliations $\F$ and $\F'$ with a tubular neighborhood of the zero section of the normal bundle of $D$ in $M$ endowed with the natural fibration.
Thus, we can consider the realization to be extended as a map from a disjoint union $K$ of closed disks contained in $D$ to $\mr{Aut}_{0}(\mb D,0)$. We can extend it to a union $K'$ of bigger disks containing $K$, being a constant automorphism of the fibres  over $\partial K'$ and consequently to the whole dicritical component $D$ using the connectedness of $\mr{Aut}_{0}(\mb D,0)$.

\medskip

\noindent $\mathbf{(1)\Rightarrow(3)}$\textbf{:} Let $g:(\underline{U},\uD)\stackrel{\sim}{\to}(\underline{U'},\uD')$ be a $\mc S$-transversely holomorphic conjugation between $(\uF,\uD)$ and $(\uF',\uD')$ and $\wt g:\wt U\to\wt U'$ a lifting to the universal coverings of $\underline{U}\setminus\uD$ and $\underline{U}'\setminus\uD'$. By  \cite[Remark~3.6.2]{monodromy} there exists a $\mc S$-$\proan$ isomorphism $h:\mc Q^{\uF}\to\mc Q^{\uF'}$ such that $(g,\wt g,h)$ is a $\mc S$-conjugation between the monodromies $\mf m^{\uF}$ and $\mf m^{\uF'}$.
Consider $\Sigma$ a $\mc S$-collection of transversals  for $\uF$ and  $\uD^{+}$ a $\uD$-extended divisor. Using
\cite[Proposition~3.6.4]{monodromy} and by  composing $g$ by a suitable $\uF'$-isotopy $\Theta_{t}$, having support on a neighborhood $W'$ of $g(\Sigma)$, we  obtain an homeomorphism $\uphi:=\Theta_1\circ g$ such that  $\Sigma':=\uphi(\Sigma)$ is a $\mc S$-collection of transversals for $\uF'$ and $\uD'$. Now we choose $\uD'^{+}$ as $\uphi(\uD^{+})$.
On the universal covering $\wt U$ we also consider the lifting $\wt\uphi$ of $\uphi$  which coincides with $\wt g$ on the complementary of $\wt W'$. Again by the same proposition, we see that $(\uphi,\wt\uphi,h)$ is a $\mc S$-conjugation of the monodromies realized over the $\mc S$-collections of transversals $\Sigma$ and $\Sigma'$. It remains to check properties (a) and (b) of Point (3).
First remark that  $\uphi$ maps isolated separatrices of $\uF$ into isolated separatrices of $\uF'$ because we have the following topological characterization:\\
\begin{enumerate}[]
\item{\it $S$ is a non-isolated separatrix if and only if there is a family $\{S_{j}\}_{j\in\mb N}$ of pairwise disjoint separatrices such that every $i,j\in\mb N$ we have that $S_{i}$ is topologically conjugated to $S$ and
 $S_{i}\cup S_{j}$ is topologically conjugated to $S\cup S_{i}$.}\\
\end{enumerate}
We deduce that $\uD'^{+}$ is a $\uD'$-extended divisor. The last assertion of Condition~(a) is trivially satisfied by the topological conjugation $\uphi$.
In~(b) equality of Camacho-Sad indices follows from Theorem~\ref{rosas} of R.~Rosas if $D$ is a nodal separatrix of $\uF$. Otherwise, $\uphi$ is transversely holomorphic in a neighborhood of $D$ and the desired equality is proved in \cite[\S7.2]{monodromy}.

\medskip

\noindent $\mathbf{(3)\Rightarrow(4)}$\textbf{:} We apply the following result which will be proved later.

\begin{lema}\label{marking}
Under the hypothesis of Point (3) there exists a germ of homeomorphism $\varphi:(M,\mc D)\to (M',\mc D')$ sending the strict transform of $\uD^{+}$ into the strict transform of $\uD'^{+}$ and a there is a lift $\wt\varphi$ of $\varphi$ to the universal coverings of the complementaries of $\cD$ and $\cD'$ satisfying the following properties:
\begin{enumerate}[(i)]
\item\label{i} at each singular point  of $\uD^{+}$ the actions of $\varphi$ and $\uphi$ on the set of local irreducible components of $\uD^{+}$ coincide;
\item\label{ii} $\varphi_{|\Sigma}=\uphi_{|\Sigma}$ and $\wt\varphi_{\wt\Sigma}=\wt\uphi_{|\wt\Sigma}$;
\item\label{iii} $\wt\varphi_{*}=\wt\uphi_{*}:\Gamma\to\Gamma'$;
\item\label{0} $\varphi$ is excellent.
\end{enumerate}
\end{lema}
Properties (\ref{ii}) et (\ref{iii}) trivially imply that $(\varphi,\wt\varphi,h)$ is a $\mc S$-conjugation between the monodromies $\mf m^{\F}$ and $\mf m^{\F'}$ realized over the $\mc S$-collections of transversals $\Sigma$ and $\Sigma'$.
From property (\ref{i}) easily follows Condition (a) of Point (4) because the strict transforms of $\overline{\uD^{+}\setminus\uD}$ and $\overline{\uD'^{+}\setminus\uD}$ allows to identify the dicritical components of $(\F,\cD)$ and $(\F',\cD')$.
Condition (b) of (4) follows from Condition (b) of (3) for local separatrices $D\subset\cD$ which are not contained in the exceptional divisor $\mc E$ of $E:M\to\uM$.
Since the dual graph of $\mc E$ is a  disjoint union of trees we can apply the same argument of \cite[\S7.3]{monodromy} to the $\mc F$-invariant part of $\mc D$ in order to obtain the equalities of all Camacho-Sad indices corresponding by $\varphi$ from those of the local separatrices of $\uF$ and $\uF'$. Finally (\ref{0}) gives (c) in Point (4).

\medskip

\begin{proof}[Proof of Lemma~\ref{marking}]
Following the notations of Section~\ref{milnor},
for $0<\eta\ll\eta'\ll\varepsilon\ll 1$ we consider an open $4$-Milnor tube $\mc T_{\eta}$ (resp. $\mc T_{\eta'}'$) associated to the divisor $\cD^{+}:=E^{-1}(\uD^{+})$ (resp. $\cD'^{+}:=E'^{-1}(\uD'^{+})$) and we denote by $\mc T$ (resp. $\mc T'$)  the image by $E$ (resp. $E'$) of its closure in the neighborhood $\overline{W}$ (resp. $\overline{W}'$) considered in Lemma~\ref{lema1}. It is worth to notice that the boundary of $\mc T$ is constituted by the closed $3$-Milnor tube $\mc M=E(\mc M_{\eta})$ and a finite union of solid tori whose boundaries are the connected components of  $\partial \mc M$.
The same property holds for $\mc T'$ and $\mc M'$. In the neighborhood of each singular point $s$ of $\uD^{+}$ (resp. $\uD'^{+}$) we consider an euclidian metric given by holomorphic coordinates. The boundaries of the closed balls $B(s,r)$ centered at $s$ with radius $r$ are transverse to $\mc M$ if $0<r\le 2\varepsilon$. We define a collar piece of $\mc T$ or $\mc M$ as
the intersection of $\overline{B(s,2\varepsilon)\setminus B(s,\varepsilon)}$ with $\mc T$ or $\mc M$. The connected components of the adherence of the complementary of the collar pieces of $\mc T$ or $\mc M$ are called essential pieces of $\mc T$ or $\mc M$. A continuous map between $\mc M$ and $\mc M'$ or $\mc T$ and $\mc T'$ will be called piece-adapted if the image of a piece is contained in a piece and the image of the boundary of a piece is contained in the boundary of a piece.
\\

\noindent\textit{First step. }
Without loss of generality we can assume that $\uphi(\mc T)\subset\mc T'$ and that any essential piece $\mc T_{s}$ of $\mc T$ containing a singular point $s$ of $\uD^{+}$ is mapped into an essential piece $\mc T'_{s'} $ of $\mc T$ containing also a singular point $s'$ of $\uD'^{+}$.
Using the local conical structures of the divisors at their singular points and
the retraction $\mc T^{*}:=\mc T\setminus\uD^{+}\to\mc M$ defined by the vector field $\xi$ considered in Section~\ref{milnor}, we can adapt the constructions of  \cite[Section~4.1]{marquage} and a variant of \cite[Lemma~4.6]{marquage} to obtain a piece-adapted continuous map $\psi_{T}:\mc T^{*}\to\mc M'\subset\mc T'^{*}$ such that
\begin{enumerate}[(a)]
\item $\psi_{T}$ is homotopic to $\uphi_{|\mc T^{*}}$ as maps from $\mc T^{*}$ into $\mc T'^{*}$ by a homotopy preserving all essential pieces associated to the singularities;
\item the restriction of $\psi_{T}$ to each connected component of a piece of $\mc M$ is a homeomorphism onto a connected component of a piece of $\mc M'$ which respect the circle fibrations considered in Lemma~\ref{lema1}.
\end{enumerate}
We define $\psi$ as the restriction of $\psi_{T}$ to $\mc M$.\\

\noindent\textit{Second step. }  For any essential piece $\mc M_{\alpha}$ of $\mc M$, the part of the proof of the main result of \cite{marquage} corresponding to Sections 4.2 to 4.4 gives us a homotopy, which preserves the boundaries, between the continuous map $\psi_{|\mc M_{\alpha}}$ and a homeomorphism $\psi_{\alpha}:\mc M_{\alpha}\to\mc M_{\alpha}'=\psi(\mc M_{\alpha})$ such that
\begin{enumerate}[(a)]
\item $\psi_{\alpha}$ extends to a homeomorphism $\Psi_{\alpha}:\mc T_{\alpha}\to\mc T_{\alpha}'$ between the corresponding pieces of $\mc T$ and $\mc T'$ containing $\mc M_{\alpha}$ and $\mc M'_{\alpha}$;
\item $\Psi_{\alpha}$ is excellent in the sense of \cite[Definition~2.5]{marquage}; in particular, the restriction of $\Psi_{\alpha}$ to $\partial \mc T_{\alpha}$ conjugates the disk fibrations considered in Lemma~\ref{lema1}.
\end{enumerate}

\vspace{1,5em}

\noindent\textit{Third step. } Using the product structure, it is straightforward to construct homotopies on the collar pieces of $\mc M$ gluing the previous homotopies defined in the essential pieces of $\mc M$. In this way we obtain a piece-adapted continuous map $\psi':\mc M\to\mc M'$ whose restriction to each essential $\mc M_{\alpha}$ coincides with the homeomorphism $\psi_{\alpha}$ but whose restriction to any collar piece is not necessarily a homeomorphism.
However, up to deforming $\psi'$ by suitable homotopies with support on the collar pieces provided by \cite[Theorem~6.1]{Waldhausen} we can assume that $\psi'$ is a piece-adapted global homeomorphism.

It remains to extend $\psi'$ to an excellent homeomorphism $\Psi$ between $\mc T$ and $\mc T'$ possessing a lifting $\wt\Psi$ to the universal coverings of $\mc T\setminus\uD$ and $\mc T'\setminus\uD'$ fulfilling properties of Lemma~\ref{marking}. On each essential piece $\mc T_{\alpha}$ we define $\Psi$ as $\Psi_{\alpha}$ constructed in second step. Since the restriction of $\Psi_{\alpha}$
to the boundary of the essential pieces conjugates the disk fibrations, we can apply the techniques given in \cite[Sections~4.4.2 and~4.4.4]{marquage} to obtain the desired extension $\Psi$. In addition, it is not difficult to modify $\Psi$ by an excellent isotopy in order to have  $\Psi_{|\Sigma}=\uphi_{|\Sigma}:\Sigma\to\Sigma'$.
Classically there exists a lifting  $\wt\Psi$ to the universal coverings $\wt{\mc T}$ and $\wt{\mc T}'$ of $\mc T\setminus\uD$ and $\mc T'\setminus\uD'$ such that $\wt\Psi_{*}=\wt\uphi_{*}:\Gamma\to\Gamma'$.

Moreover, since the restriction of $\uphi$ and $\Psi$ to each singular piece $\mc M_{\alpha}$ are related by a homotopy localized in $\mc M_{\alpha}$ it follows that
$$
\wt\uphi_{\bullet}=\wt\Psi_{\bullet}:\pi_{0}(\wt{\mc T}_{\alpha})\to\pi_{0}(\wt{\mc T}'_{\alpha}).
$$
Thanks to this last equality we can apply the procedure described in \cite[Section~8.4]{monodromy} in order to modify $(\Psi,\wt\Psi)$ by Dehn twists to obtain a new pair $(\varphi,\wt\varphi)$ which satisfy the same properties~(\ref{iii}) and~(\ref{0}) and fulfills also the equality
$$
\wt\uphi_{\bullet}=\wt\varphi_{\bullet}:\pi_{0}(\wt\Sigma)\to\pi_{0}(\wt\Sigma').
$$
Up to making an additional Dehn twist if necessary we obtain that $\wt\varphi_{|\Sigma}=\wt\uphi_{|\wt\Sigma}$, showing Property~(\ref{ii}).

Since Property~(\ref{i}) follows from our construction, the proof of the lemma is achieved.
\end{proof}
\color{black}

\end{document}